\documentclass[a4paper,english,10pt]{amsart}
\usepackage[utf8x]{inputenc}
\usepackage{babel}
\usepackage{amsfonts}
\usepackage{amsmath}
\usepackage{amssymb}
\usepackage{amsthm}
\usepackage{graphicx}
\usepackage[all]{xy}
\usepackage{textcomp}

\newtheorem{thm}{Theorem}[section]  

\newtheorem{lem}[thm]{Lemma}
\newtheorem{defi}[thm]{Definition}
\newtheorem{prop}[thm]{Proposition}
\newtheorem{es}[thm]{Example}
\newtheorem{rem}[thm]{Remark}

\title{ Finite Morphisms between projective Varieties and Skeleta}
\author{John Welliaveetil }

\begin{document}

\maketitle

\textbf{Address} : John Welliaveetil, Institut de Mathématiques de Jussieu, UMR 7586 du CNRS, Université Pierre
et Marie Curie, Paris, France. \\

\textbf{E-mail address} : welliaveetil@gmail.com

\begin{abstract}
     In this paper we study finite surjective morphisms between irreducible projective varieties in terms of the morphisms
they induce between the respective analytifications. The background for the principal result is as follows. 
Let $V'$ and $V$ be irreducible, projective varieties over 
an algebraically closed, non-Archimedean real valued field $k$ with $V$ normal and $\phi: V' \to V$ be a finite surjective morphism .
Let $x \in V^{\mathrm{an}}(L)$ (Remark 1.1), where $L/k$ is an algebraically 
closed complete non-Archimedean real valued field extension.
 We associate canonically 
 to $x$ an $L$-point of the space $(V \times_k L)^{\mathrm{an}}$ which lies on the fibre over $x$ and denote this point 
$x_L$. 
 The embedding of $V$ into some $n$-dimensional projective space
defines in a natural way a family 
of open neighbourhoods of $x_L$ contained in $(V \times_k L)^{\mathrm{an}}$.
Each element of this family is parametrized by an $(n+1)^2$-tuple which quantifies its size. 
Of particular interest to us will be those neighbourhoods $O$ whose 
preimage for the morphism $(\phi \times \mathrm{id}_L)^{\mathrm{an}}$ decomposes into the disjoint union of
 homeomorphic copies of $O$ 
via $(\phi \times \mathrm{id}_L)^{\mathrm{an}}$. Let $\mathcal{G}_{x_L}$
 denote the sub collection of elements of this form. 
 Theorem 1.7 shows that there exists a deformation retraction of 
the space $V^{\mathrm{an}}$ onto a finite simplicial complex such that along the fibres of the retraction the 
size of the largest element belonging to $\mathcal{G}_{x_L}$ is constant.   
\end{abstract}


This research was funded by the ERC Advanced Grant NMNAG.
 
\tableofcontents

\section{ Introduction }
    In this paper we study finite surjective morphisms between irreducible projective varieties 
over non-Archimedean real valued fields in terms of the morphisms they induce between the analytifications of the varieties. 
As topological spaces the analytifications are Hausdorff, compact and path-connected. 
 The theorem we prove implies that the induced morphism when viewed as a continuous map between topological spaces 
admits a certain uniform behaviour. Before stating the theorem 
in full generality, we provide its motivation by considering the case of a finite endomorphism of the projective line.

    Let $k$ be an algebraically closed, complete non-Archimedean real valued field whose value group $|k^*|$ contains at least two elements 
    and is a sub group of $(\mathbb{R}_{>0},*)$. 
 Let $\mathbb{P}_k^{1,\mathrm{an}}$ be the Berkovich analytification of the projective line 
$\mathbb{P}^1_k$. The analytification $\mathbb{P}^{1,\mathrm{an}}_k$ allows us  
 to use the valuative topology provided by the field to study an algebraic endomorphism. For a point 
$x  \in \mathbb{P}_k^1(k) \subset \mathbb{P}^{1,\mathrm{an}}_k$, we have the notion of a Berkovich closed disk or Berkovich 
open ball centred at $x$ within the space $\mathbb{P}^{1,\mathrm{an}}_k$ which contains the
naive closed or open disk
around $x$. By the naive closed (open) disk around $x \in k$ of radius $r \in \mathbb{R}_{> 0}$, we mean the set
$\{y \in k| |y-x| \leq r\}$ ($\{y \in k| |y-x| < r\}$). As opposed to their naive counterparts, the Berkovich open and closed disks are locally compact 
and contractible. 

   Let $\phi : \mathbb{P}^1_k \to \mathbb{P}^1_k$ be a finite morphism. For a
complete non-Archimedean real valued algebraically closed field extension $L/k$, let
$\phi_L$ denote the morphism 
\begin{align*}
  \phi \times \mathrm{id}_L : \mathbb{P}^1_k \times_k \mathrm{Spec}(L) \to \mathbb{P}^1_k \times_k \mathrm{Spec}(L) .
\end{align*}
     The analytification of a $k$-variety of finite type (cf. Section 2.1)
is functorial and hence endomorphisms of the projective line will induce endomorphisms of its analytification. 
That is, the morphism $\phi$ induces a morphism
\begin{align*}
  \phi^{\mathrm{an}} : \mathbb{P}^{1,\mathrm{an}}_k \to \mathbb{P}^{1,\mathrm{an}}_k.
\end{align*}
   The morphism $\phi_L^{\mathrm{an}}$ is similarly defined and it is to be noted that 
$\phi_L^{\mathrm{an}} = \phi^{\mathrm{an}} \times \mathrm{id}_L^{\mathrm{an}}$.  

   Our reason for introducing the morphism $\phi_L$ for a   
complete non-Archimedean real valued algebraically closed field extension $L/k$
is to deal with all points of the 
analytification of the
projective line over $k$ and not just those points for which $\mathcal{H}(x) = k$ (cf. 2.1.1). 
When discussing the points of the 
analytification of a $k$-variety, 
we make use of the description outlined in Section 2.1. 
Let $x \in \mathbb{P}^{1,\mathrm{an}}_k(L)$ i.e. $\mathcal{H}(x) \subset L$.
The image of $x \in \mathbb{P}^{1,\mathrm{an}}_k$ for the morphism 
$\pi : \mathbb{P}^{1,\mathrm{an}}_k \to \mathbb{P}^1_k$ (cf. 2.1) 
is an $L$-point of $ \mathbb{P}^1_k$. We abuse notation and refer to this point as $x$ as well. 
The pair $x : \mathrm{Spec}(L) \to \mathbb{P}^{1}_k$ and 
$\mathrm{id}_L : \mathrm{Spec}(L) \to \mathrm{Spec}(L)$ defines a closed point of the 
variety $\mathbb{P}^1_k \times_k \mathrm{Spec}(L)$ which we denote $x_L$. 
The following remark generalizes this construction. 

\begin{rem} Let $V$ be a $k$-variety. The notation $x \in V^{\mathrm{an}}(L)$
  will be used to mean that the image of the point $x$
   for the morphism $V^{\mathrm{an}} \to V$ which is 
   defined in Section 2.1 is an $L$-point.
   We abuse notation and refer to this point as $x$ as well. 
  The pair $x : \mathrm{Spec}(L) \to V$ and 
$\mathrm{id}_L : \mathrm{Spec}(L) \to \mathrm{Spec}(L)$ defines a closed point of the 
variety $V \times_k \mathrm{Spec}(L)$ which we denote $x_L$. This construction will be referred to frequently in what follows. 
Note that the notation $x \in V^{\mathrm{an}}(L)$ defined in this manner is
 equivalent to the following inclusion of non-Archimedean real valued complete fields,  $\mathcal{H}(x) \subset L$. The field 
$\mathcal{H}(x)$ is the non - Archimedean geometry analogue of the residue field of a point in algebraic geometry. 
It is
 defined in 2.1.1. 
 \end{rem}   
    
  We now introduce the theorem concerning finite endomorphisms of the projective line over $k$.        
\begin{thm}
  Let $\phi : \mathbb{P}^1_k \to \mathbb{P}^1_k$ be a finite morphism.
  Let $x \in \mathbb{P}^{1,\mathrm{an}}_k$ and $L/k$ be
any complete non-Archimedean real valued algebraically closed field extension of $k$ such that 
$x \in \mathbb{P}^{1,\mathrm{an}}_k(L)$. Let $f(x)$ be
the minimum of $1$ and the radius of the largest Berkovich open ball $B \subset \mathbb{P}^{1,\mathrm{an}}_L$ around 
$x_{L} \in \mathbb{P}^{1}_L(L) \subset \mathbb{P}^{1,\mathrm{an}}_L$
 whose preimage under 
$\phi_{L}^{\mathrm{an}}$ is the disjoint union of 
homeomorphic copies of $B$ via $\phi_{L}^{\mathrm{an}}$.
 The function  
$f : \mathbb{P}^{1,\mathrm{an}}_k \to \mathbb{R}_{\geq 0}$ is not identically zero and well defined. 
 There exists a finite simplicial complex
$\Upsilon \subset \mathbb{P}^{1,\mathrm{an}}_k$, a generalised real interval $I := [i,e]$ and a deformation retraction 
\begin{align*}
  \psi : I \times \mathbb{P}_k^{1,\mathrm{an}} \to \mathbb{P}_k^{1,\mathrm{an}}
\end{align*}
 such that $\psi(e,\mathbb{P}_k^{1,\mathrm{an}}) = \Upsilon$ and the function $f$ is constant on the fibres of this retraction i.e. 
 for every $x \in \mathbb{P}_k^{1,\mathrm{an}}$ we have that $f(x) = f(\psi(e,x))$. Furthermore, the function $log_c|f|$ is 
 piecewise linear when restricted to $\Upsilon$ where $0 < c < 1$ is a real number.       
\end{thm}
   
   \begin{rem}  We fix the real number $c$ which appears in the portion of the theorem above concerning piecewise linearity and 
     hence forth write $\mathrm{log}(|f|)$ in place of $\mathrm{log}_c(|f|)$. 
     \end{rem} 
     The notion of a generalised interval is discussed in Section 3.9 \cite{loes}.
      We now provide an example which illustrates the behaviour of the function $f$ in Theorem $1.2$ clearly. 
      
\begin{es}
   Let $k$ be an algebraically closed    
 complete non-trivially valued non-Archimedean field which is of characteristic $p$.
Consider the morphism $\phi : \mathbb{P}_k^1 \to \mathbb{P}_k^1$ given by $z \mapsto z^p - z$.                   
Let $L/k$ be a non-Archimedean real valued field extension. 
The morphism $\phi_L$ is \'etale over every point other than $\infty$.
 Furthermore it can be shown that $f(x) = 1$ if 
$x \neq \infty$ and $0$ at $\infty$. Let
$\Upsilon$ be a 
 finite graph containing the point $\infty$ and which contains at least one other point.  
 Since there is a deformation retraction of 
  $\mathbb{P}_k^{1,\mathrm{an}}$ onto any finite sub-graph, it follows that 
  there exists a deformation retraction 
 \begin{align*}
\psi : I \times \mathbb{P}_k^{1,\mathrm{an}} \to \mathbb{P}_k^{1,\mathrm{an}}
\end{align*}
   such that the function $f$ is 
constant along the fibres of the retraction.  
 \end{es}  
  
     Our first goal is to generalize Theorem 1.2 to the case of finite surjective morphisms between irreducible, projective varieties. 
A problem standing in the way of any attempt at a generalisation is that there is no intrinsic notion of
 an open disk in $V^{\mathrm{an}}$ if $V$ is a projective $k$-variety of finite type.
However, as $V$ is projective there exists a closed immersion 
$V \hookrightarrow \mathbb{P}^n_k$ for some $n \in \mathbb{N}$. We identify $V$ with its image under the closed immersion.  
The space $\mathbb{P}^{n,\mathrm{an}}_k$ can be equipped with a finite formal cover [\cite{berk}, Section 4.3] 
such that each element of this cover is isomorphic to the $n$-dimensional Berkovich closed disk $\mathcal{M}(k\{T_1,\ldots,T_n\})$.
Let $\{A_i\}_i$ denote this cover. 
The intersection of the
 elements of the formal cover with the image of the 
immersion $V^{\mathrm{an}} \hookrightarrow \mathbb{P}^{n,\mathrm{an}}_k$ 
defines a formal cover of the space $V^{\mathrm{an}}$, namely 
$\{A_i \cap V^{\mathrm{an}}\}_i $. Furthermore, for a non-Archimedean real valued algebraically closed field extension 
$L/k$ 
the construction extends to the analytic space $(V_L)^{\mathrm{an}} := (V \times_k L)^{\mathrm{an}}$. 
The $n$ - dimensional Berkovich open balls contained in $\mathbb{P}^{n,\mathrm{an}}_L$ allow us to generalise 
Theorem 1.2. We now provide a sketch of the details of this construction.  
 
Let $L/k$ be a non-Archimedean real valued algebraically closed field extension and
$x \in \mathbb{P}^{n,\mathrm{an}}_L(L)$ (Remark 1.1). Let $x_L \in A_{i,L}$ for some $i$. 
Associated to the point $x_L$
 is a collection of open   
neighbourhoods in $\mathbb{P}^{n,\mathrm{an}}_L(L)$, namely the Berkovich  
open balls around $x_L$ contained in $A_{i,L}$.   
We denote this family of open neighbourhoods $\mathcal{O}_{x_L}$. 
Let $G \in \mathcal{O}_{x_L}$. 
 For every $j$ such that  
$x_L$ belongs to $A_{j,L}$, it can be checked that $G$ is a Berkovich open ball in 
$A_{j,L}$ as well.
This implies that the family $\mathcal{O}_{x_L}$ is defined independent of the 
element of the affinoid chart which contains $x_L$ and is hence well defined. 
We now define a poly radius associated to the elements of $\mathcal{O}_{x_L}$. 
Let $x_L$ have homogenous coordinates $[x_{1,L} : \ldots : x_{n+1,L}]$ and let
 $W \in \mathcal{O}_{x_L}$. We can associate an $(n+1)^2$-tuple denoted $h_L(W)$
 to the open neighbourhood $W$ as follows. If for an index $t$, $x_{L} \in A_{t,L}$ then let 
$\mathbf{r}_t = (r_1,\ldots,r_{n+1})$ be such that $r_t = 1$ and the Berkovich open ball 
$B_{W,t}$ is defined by the equations $|(T_j/T_t - x_{j,L}/x_{t,L})(p)| < r_j$ for $j \neq t$.
 If on the other hand $x_{L}$ does not belong to $A_{t,L}$ then let $\mathbf{r}_t = (1,\ldots,1)$. 
 We define $h_L(W) := (\mathbf{r}_t)_t$. Let 
 $\mathcal{O}_{L} := \bigcup_{x \in \mathbb{P}^{n,\mathrm{an}}_k(L)} \mathcal{O}_{x_L}$.
 We can extend the above construction to define a function 
 $h_L : \mathcal{O}_{L} \to \mathbb{R}_{\geq 0}$. 
Note that the
sets $\mathcal{O}_{x_L}$ depend on the affine chart chosen for $\mathbb{P}^{n}$.
In the case of $\mathbb{P}^1_k$ with $x \in \mathbb{P}^{1,\mathrm{an}}_k(L)$, 
 the set $\mathcal{O}_{x_L}$ associated to the construction above
 is discussed in Section 2.2.1. In what follows we explain how the family $\mathcal{O}_L$ can be ordered. 
 
 \begin{rem}   
    We now introduce a collection $S$ of functions from $\mathbb{R}_{> 0}^{(n+1)^2}$ to $\mathbb{R}_{> 0}$.
Let $g \in S$ if and only if
\begin{enumerate}
\item The function $g$ is continuous. 
 \item If $(r_{i,j})_{i,j}$ and $(s_{i,j})_{i,j}$ are $(n+1)^2$-tuples such that 
$r_{i,j} \leq s_{i,j}$ then $g((r_{i,j})_{i,j}) \leq g((s_{i,j})_{i,j})$.
\item $g$ is a definable function (in the model theoretic sense) in the language of Ordered Abelian groups. \end{enumerate}

     Let $g \in S$. We can extend the function $g$ so that it defines a function 
     $\mathbb{R}_{\geq 0}^{(n+1)^2} \to \mathbb{R}_{\geq 0}$ by mapping any tuple 
     $\mathbf{r} = (r_{i,j})_{i,j} \in \mathbb{R}_{\geq 0}^{(n+1)^2}$ such that $r_{l,m} = 0$ for some $l,m$ to 
     $0$. An element of the class of functions $S$ can thus be extended so that it defines a total ordering on
     $\mathbb{R}_{\geq 0}^{(n+1)^2}$ which extends the partial ordering 
   given by : $(r_{i,j})_{i,j} \leq (s_{i,j})_{i,j}$ if 
$r_{i,j} \leq s_{i,j}$, where $(r_{i,j})_{i,j}$ and $(s_{i,j})_{i,j}$  
   are $\mathbb{R}_{\geq 0}^{(n+1)^2}$ tuples. 
\end{rem}
   
   Let
  $g \in S$. By Lemma 2.8, the function 
$g \circ h_L : \mathcal{O}_{L} \to  \mathbb{R}_{\geq 0}$ has the following property.  
 If 
$O_1, O_2 \in \mathcal{O}_{L}$ such that $O_1 \subseteq O_2$ then     
$(g \circ h_L)(O_1) \leq (g \circ h_L)(O_2)$. The functions $g \in S$ hence allow us to quantify the size of elements belonging to $\mathcal{O}_{L}$. 

   We now provide an equivalent form of Theorem $1.2$ which we generalize. We begin by motivating 
the reformulation. The goal of Theorem $1.2$ is to prove the existence of a finite 
simplicial complex $\Upsilon$ contained in $\mathbb{P}^{1,\mathrm{an}}_k$ such that the function $f$ is constant along the fibres of the retraction 
morphism $\psi(e,\_)$. Let us assume that Theorem $1.2$ is true. We define a function 
$M : \Upsilon \to \mathbb{R}_{\geq 0}$ as follows. Let $\gamma \in \Upsilon$ and 
$x \in \mathbb{P}^{1,\mathrm{an}}_k$ be any point for which $\psi(e,x) = \gamma$. 
We set 
$M(\gamma) := f(x)$.  
Let $L/k$ be any complete non-Archimedean real valued algebraically closed field extension such that
$x \in \mathbb{P}^{1,\mathrm{an}}_k(L)$. By definition $M(\gamma)$ is
the   
minimum of $1$ and the 
radius of the largest Berkovich open ball in $\mathbb{P}^{1,\mathrm{an}}_L$ around $x_L$
whose preimage is the disjoint union of homeomorphic copies of itself for the morphism $\phi_L^{\mathrm{an}}$.
The function $M$ is well defined since we assumed Theorem $1.2$ is true.  
Hence a suitable restatement of $1.2$ is 
the following theorem. 
  
\begin{thm}
    Let $\phi : \mathbb{P}^1_k \to \mathbb{P}^1_k$ be a finite morphism. 
    There exists a generalised real interval $I := [i,e]$ and a deformation retraction 
\begin{align*}
  \psi : I \times \mathbb{P}_k^{1,\mathrm{an}} \to \mathbb{P}_k^{1,\mathrm{an}}
\end{align*}
which satisfies the following properties. 
\begin{enumerate}
 \item The image $\psi(e,\mathbb{P}_k^{1,\mathrm{an}})$ of the deformation retraction $\psi$ is a finite simplicial complex. 
Let $\Upsilon$ denote this finite simplicial complex. 
 \item  There exists
a well defined function $M : \Upsilon \to [0,1]$
which satisfies the following conditions. The function $M$ is not identically zero and $log(M)$ is piecewise linear. 
Let $\gamma \in \Upsilon$ such that $M(\gamma) > 0$ and $x \in \psi(e,\_)^{-1}(\gamma)$.
Let $L/k$ be any complete non-Archimedean real valued algebraically closed field extension such that
$x \in \mathbb{P}^{1,\mathrm{an}}_k(L)$. Then the following are true. 
\begin{enumerate}
  \item The
 preimage under the morphism $\phi_L^{\mathrm{an}}$ of the Berkovich open ball 
$B(x_L,M(\gamma)) \subset \mathbb{P}^{1,\mathrm{an}}_L$ around $x_L$ of radius $M(\gamma)$ 
 decomposes into the disjoint union of Berkovich open balls each homeomorphic to
$B(x_L,M(\gamma))$ via the morphism $\phi_L^{\mathrm{an}}$. 
\item Let $O$ be any other Berkovich open ball around $x_L$ whose radius is less than or equal to $1$ such that its 
preimage under the morphism $\phi_L^{\mathrm{an}}$ decomposes into the disjoint union of Berkovich open balls each 
homeomorphic to $O$ via $\phi_L^{\mathrm{an}}$. Then the radius of $O$ must be less than or equal to $M(\gamma)$. 
\end{enumerate}
\end{enumerate}
 \end{thm}       

   We will show that Theorems 1.2 and 1.6 are equivalent in Section 6. 
We now state a theorem which in Section 6 we will show to be a generalisation of Theorem 1.2. 
Let $\phi : V' \to V$ be a finite surjective
morphism between irreducible, projective varieties of finite type over $k$.
For a complete non-Archimedean real valued algebraically closed field extension $L/k$, let 
 $\phi_L$ denote the morphism 
\begin{align*}
  \phi \times \mathrm{id}_L : V \times_k L \to V \times_k L .
\end{align*}
  As in the case of 
  $\mathbb{P}^1_k$, we write $\phi^{\mathrm{an}} : V'^{\mathrm{an}} \to V^{\mathrm{an}}$ for the induced morphism between the respective analytifications. 
The morphism $\phi_L^{\mathrm{an}}$ is similarly defined and it is to be noted that 
$\phi_L^{\mathrm{an}} = \phi^{\mathrm{an}} \times \mathrm{id}_L^{\mathrm{an}}$.  We fix an embedding
$V \hookrightarrow \mathbb{P}^n_k$ and an affine chart of $\mathbb{P}^{n}$.   
  The theorem will be
stated in terms of the sets $\mathcal{O}_{x_L}$ and the functions $h_L$ and $g$ described above.    
 
\begin{thm}
    Let $\phi : V' \to V$ be a finite surjective morphism between irreducible, projective varieties with $V$ normal. 
 Let $g \in S$. There exists a generalized real interval $I := [i,e]$
  and a deformation retraction 
\begin{align*}
  \psi : I \times V^{\mathrm{an}} \to V^{\mathrm{an}}
\end{align*}
  which satisfies the following properties. 
\begin{enumerate}
  \item The image $\psi(e,V^{\mathrm{an}})$ of the deformation retraction $\psi$ is a finite simplicial complex which we denote $\Upsilon_g$.
   \item There exists a well defined function $M_g : \Upsilon_g \to \mathbb{R}_{≥ 0}$ which satisfies the following conditions. 
The function $M_g$ is not identically zero. Let 
$\gamma \in \Upsilon_g$ be a point on the finite simplicial complex for which $M_g(\gamma) \neq 0$ and 
$x \in \psi(e,\_)^{-1}(\gamma)$. 
Let $L/k$ be any complete non-Archimedean real valued algebraically closed field extension such that 
$x \in V^{\mathrm{an}}(L)$. There exists $W \in (g \circ h_L)^{-1}(M_g(\gamma)) \cap \mathcal{O}_{x_L}$ such that the open set  
 $(\phi^{\mathrm{an}}_L)^{-1}(W \cap V_L^{\mathrm{an}})  \subset V_L'^{\mathrm{an}}$ decomposes into the disjoint union of open sets, each
 homeomorphic to $W \cap V_L^{\mathrm{an}}$ via $\phi_L^{\mathrm{an}}$. Furthermore, let
 $O \in \mathcal{O}_{x_L}$ be such that the preimage of $O \cap V_L^{\mathrm{an}}$ under $\phi_L^{\mathrm{an}}$
decomposes into 
the disjoint union of open sets in $V_L'^{\mathrm{an}}$, each 
 homeomorphic to $O$ via the morphism $\phi_L^{\mathrm{an}}$. Then $(g \circ h_L)(O) \leq M_g(\gamma)$.
 Lastly, the function $\mathrm{log}(M_g)$ is piecewise linear on $\Upsilon_g$. 
 \end{enumerate}
\end{thm}
   
   \begin{rem}
       The second property we require the function $M_g$ to satisfy involves choosing 
  a point $x_L$ over $x$ which is an $L$-point 
  of $V_L^{\mathrm{an}}$ and then requiring that $M_g$ fulfill a condition concerning $x_L$. 
  It may hence seem that the function is dependent on the field $L$ chosen. 
  However, showing that the function $M_g$ is well defined 
 will in particular imply that its value depends solely on $x$ and not on the points $x_L$ for $L/k$.    
  \end{rem}
  
   It is worth mentioning that the result stated above when applied to smooth Berkovich analytic curves 
   over a field of characteristic 
zero bears some similarity with theorems proved in \cite{PP} and \cite{Bal}. In \cite{Bal}, 
the author - F. Baldassarri studies a system of differential equations defined over an analytic domain of the 
affine line over a non-Archimedean real valued field of characteristic zero. 
More precisely, let $k$ be a non-Archimedean field of characteristic zero and 
$X$ be a relatively compact analytic domain of the affine line $\mathbb{A}^{1,\mathrm{an}}_k$.
Let
\begin{align*}
 \Sigma : dy/dT = Gy
\end{align*} 
be a system of linear differential equations where $G$ is a $\mu \times \mu$ matrix of $k$-analytic functions on $X$. If $x ∈ X$
is a $k$-rational point, let $R(x) = R(x,\Sigma)$ denote the radius of the maximal open disk in 
$X$
with center at $x$ on which all solutions of $\Sigma$ converge.
The author shows that the function $R$ is continuous. Also, he illustrates how when $X = \mathbb{A}^{1,\mathrm{an}}_k$   
there exists a finite graph $\Gamma \subset \mathbb{A}^{1,\mathrm{an}}_k$ which \emph{controls} 
 the behaviour of the function $R$. Since $\mathbb{A}^{1,\mathrm{an}}_k$ retracts to any of its finite subgraphs, this means that 
the function $R$ is constant along the fibres of the retraction on $\Gamma$.
In the paper the control by a finite graph is illustrated by an example. 
If one preserves the 
restrictions on the field $k$ and considers the case of
a system of differential equations defined instead over a
 smooth Berkovich curve then a similar result holds true. 
In \cite{PP}, the authors - Poineau and Pulita prove that
associated to a system of differential equations over a smooth Berkovich curve,
 there exists a locally finite graph contained in the curve
 and a retraction of the curve onto it such that the radius of convergence function is constant along the fibres of the
 retraction. 
The result we prove in this paper and the results of Poineau-Pulita and Baldassari 
show that the behaviour of certain functions of interest are controlled by finite simplicial complexes associated to them.  
       
   The goal of this paper is to prove Theorem $1.7$. In Section $2$ we will discuss results and various concepts
from Model Theory and the theory of Berkovich spaces which we will require. We will also
 construct the set $\mathcal{O}_{x_L}$ for any $x \in V^{\mathrm{an}}(L)$ where $L/k$ 
is a complete non-Archimedean real valued algebraically closed field extension.
In Section 4, we prove a version of the main theorem for Hrushovski - Loeser spaces and 
derive 1.7 from this result in Section 5.  
 \\ 
 
 \noindent \emph{Notation} : To prove the main theorem of this paper, we require techniques from Model theory where it is standard to write the value group of 
 a non-Archimedean valued field additively. However, when discussing objects from Berkovich geometry 
 such as affinoid algebras and the reduction morphism it is standard to endow the value group with a  multiplicative structure as 
 it aids in intuition. Instead of resolving this dichotomy in notation, we preserve both notation and eliminate ambiguity by specifying 
 at every instance the structure of the value group, i.e. whether we look at it additively or multiplicatively. 

 \section{Preliminaries}
\subsection{Berkovich spaces}
     Let $V$ be an irreducible, projective 
$k$-variety. By $k$-variety we mean a separated
 $k$-scheme of finite type. Associated functorially to $V$ is a Berkovich analytic space $V^{\mathrm{an}}$. 
We examine this notion in more detail.
We write the value group of the field $k$ multiplicatively 
in this section as well as in 2.2 where we discuss results from Berkovich geometry.
    
   Let $X$ be a scheme which is locally of finite type over $k$. Let $k-an$ denote the category of
$k$-analytic spaces [\cite{berk2}, 1.2.4], $Set$ denote the category of sets and $Sch_{lft/k}$ denote the category of schemes which
are locally 
of finite type over $k$. 
   We define a functor 
\begin{align*}        
   F &:  k-an \to Set \\
     &    Y \mapsto Hom(Y,X)
 \end{align*}
where $Hom(Y,X)$ is the set of morphisms of $k$-ringed spaces. The following theorem defines the space $X^{\mathrm{an}}$.   

\begin{thm} [\cite{berk}, 1.2.4]
   The functor $F$ is representable by a $k$-analytic space $X^{\mathrm{an}}$ and a morphism 
$\pi : X^{\mathrm{an}} \to X$. For any non-Archimedean complete real valued field $K$ extending $k$, there is a 
bijection $X^{\mathrm{an}}(K) \to X(K)$. Furthermore, the map $\pi$ is surjective. 
\end{thm}

   The associated $k$-analytic space $X^{\mathrm{an}}$ is \emph{good} by which we mean that for every point $x \in X^{\mathrm{an}}$ there exists a neighbourhood of $x$
isomorphic
to an affinoid space. Theorem 2.1 implies the existence of a well defined functor
\begin{align*}
   ()^{\mathrm{an}} :& Sch_{lft/k} \to  \mbox{\emph{good} } k-an \\
           &     X        \mapsto  X^{\mathrm{an}}.  
\end{align*}
     
    As a set $X^{\mathrm{an}}$ is the collection of pairs $\{(x,\eta)\}$ where $x$ is a scheme 
theoretic point of $X$ and $\eta$ is a rank one valuation on the residue field $k(x)$ which extends the 
valuation on the field $k$. We endow this set with a topology as follows. A pre-basic open set 
is a set
 of the form $\{(x,\eta) \in U^{\mathrm{an}}||f(\eta)| \in W\}$, where $U$ is an open subvariety of $X$ with $f \in O_X(U)$, 
$W$ is an open subspace of $\mathbb{R}_\geq 0$ and $|f(\eta)|$ is the image of $f$ in the residue field $k(x)$ evaluated at 
$\eta$. A basic open set is any set which is equal to the intersection of a finite number of pre-basic open sets.    

   The description above implies that
 every point of the analytification $X^{\mathrm{an}}$ is associated to a complete non-Archimedean algebraically closed real valued field
extension $L/k$ and a $k$-morphism $\mathrm{Spec}(L) \to X$.
We explore this idea further. Let
$x_1 : \mathrm{Spec}(L_1) \to X$ and 
$x_2 : \mathrm{Spec}(L_2) \to X$ be a pair of $k$-morphisms where 
$L_1/k$ and $L_2/k$ are complete non-Archimedean real valued algebraically closed
field extensions. We set $x_1 \sim x_2$ if there exists a  
complete non-Archimedean real valued algebraically closed field extension $M$ of both 
$L_1$ and $L_2$ and a $k$-morphism $x_3: \mathrm{Spec}(M) \to X$ such that 
\begin{align*}
   x_3 = x_i \circ (\mathrm{Spec}(M) \to \mathrm{Spec}(L_i))  
\end{align*}
   for $i =1,2$ and where $\mathrm{Spec}(M) \to \mathrm{Spec}(L_i)$ is induced by the inclusion
$L_i \hookrightarrow M$.  

  We have the following equality. 
\begin{align}
  X^{\mathrm{an}} = \{\bigcup_{L/k} Hom_{\mathrm{Spec}(k)}(\mathrm{Spec}(L),X)\}/\sim .
\end{align}
  The union in the above equation is taken over all complete non-Archimedean real valued algebraically closed field extensions $L$ of $k$. 
Using (1), the morphism $\pi : X^{\mathrm{an}} \to X$ which appears in the statement of Theorem 2.1 can be 
defined explicitly as follows. 
\begin{align*}
  \pi &: X^{\mathrm{an}} \to X \\
       & [x:\mathrm{Spec}(L) \to X] \mapsto x:\mathrm{Spec}(L) \to X    
 \end{align*}
    where $[x:\mathrm{Spec}(L) \to X)]$ denotes the equivalence class of the point 
$x$ for the relation $\sim$ defined above.    
 \begin{es}
  \emph{ If $V = \mathbb{A}^n_k$ then as a set $V^{\mathrm{an}}$ is simply 
the collection of multiplicative seminorms on the polynomial ring $k[T_1,\ldots,T_n]$ which restrict to the given valuation on  
$k$. We use $\mathbb{A}^{n,\mathrm{an}}_k$ to denote the analytification of affine $n$-space over $k$. Similarly,  
as a set the analytification of a Zariski closed subset of $\mathbb{A}^n_k$ whose ring of 
regular functions is $k[T_1,\ldots,T_n]/I$ is the collection of multiplicative seminorms on 
$k[T_1,\ldots,T_n]/I$ which restrict to the given valuation on $k$. Multiplicative seminorms are defined in section 2.1.1}  
\end{es}
   
    We now discuss the notion of the reduction map associated to a $k$-affinoid space.
 The reduction map and the $\tilde{k}$-scheme associated to an affinoid space will be 
of use when discussing  
 formal covers in the following subsection. 

\subsubsection{The Reduction Morphism}
    Let $\mathcal{A}$ be a commutative ring with unity. A non-Archimedean seminorm on $\mathcal{A}$ is a function 
$|.| : \mathcal{A} \to \mathbb{R}_{\geq 0}$ such that $|0| = 0$, $|1| =1$
and if $f,g \in \mathcal{A}$ then
$|f - g| \leq \mathrm{max}\{|f|,|g|\}$ and $|f.g| \leq |f|.|g|$. A seminorm is called a norm if for any 
$a \in \mathcal{A}$ for which $|a| = 0$ then $a =0$. If $|.|$ is a seminorm on $\mathcal{A}$ such that 
$|f.g| = |f|.|g|$ then it is called a multiplicative seminorm. A norm which is multiplicative is called a valuation.      
     
  Let $|.|$ be a norm on the ring $\mathcal{A}$. The pair $(\mathcal{A},|.|)$ is a Banach ring if 
$\mathcal{A}$ is complete for the norm $|.|$. When there is no ambiguity concerning the norm on 
$\mathcal{A}$ we suppress notation and refer to the Banach ring as just $\mathcal{A}$.  
  A bounded homomorphism
of Banach rings $\phi : \mathcal{A} \to \mathcal{B}$ is a homomorphism of rings with the additional property that there exists a 
constant
$C \in \mathbb{R}$ such that for every $a \in \mathcal{A}$, we have the inequality $|\phi(a)|_\mathcal{B} \leq C.|a|_\mathcal{A}$. 
 
    Let $\mathcal{A}$ be a Banach ring. We define its spectrum $\mathcal{M}(\mathcal{A})$ to
be the set of bounded multiplicative seminorms on $\mathcal{A}$ provided with the weakest topology such
that for every $f \in \mathcal{A}$ the
 real valued function $f : \mathcal{M}(\mathcal{A}) \to \mathbb{R}_{\geq 0}$ defined by $x \mapsto |f(x)|$
 is continuous [\cite{berk}, 1.2].
A bounded homomorphism of commutative Banach rings $\phi : \mathcal{A} \to \mathcal{B}$ will induce 
a continuous morphism $\mathcal{M}(\mathcal{B}) \to \mathcal{M}(\mathcal{A})$. 

    The spectral norm of $\mathcal{A}$ is denoted $\rho$ and defined as follows. 
    Let $f \in \mathcal{A}$. We set
    $\rho(f) := \mathrm{sup}_{x \in \mathcal{M}(\mathcal{A})} \{|f(x)|\}$.  
 
   If $x \in \mathcal{M}(\mathcal{A})$ and $f \in \mathcal{A}$ then
we write $|f(x)|$ for the value of $f$ at $x$. 
It can be checked that $\mathrm{Ker}(x) := \{f \in \mathcal{A} | |f(x)| = 0 \}$ is a prime ideal.  
Let $\mathcal{H}(x)$ denote the completion of the field of fractions 
of $\mathcal{A}/(\mathrm{Ker}(x))$ with respect to the norm induced on this quotient by $x$. 
            
   The set $\mathcal{A}^\circ := \{x \in \mathcal{A}| \rho(x) \leq 1\}$ is a complete sub ring  
of $\mathcal{A}$ in which $\mathcal{A}^{\circ \circ} := \{x \in \mathcal{A}| \rho(x) < 1\}$ is an ideal. We will set
$\tilde{\mathcal{A}} := \mathcal{A}^\circ/ \mathcal{A}^{\circ \circ}$. A bounded homomorphism of
Banach rings $\mathcal{A} \to \mathcal{B}$ will induce a morphism $\tilde{\mathcal{A}} \to 
\tilde{\mathcal{B}}$. Note that if $\mathcal{A}$ is a field then $\tilde{\mathcal{A}}$ is a field as well. 

  For every $x \in \mathcal{M}(\mathcal{A})$, we have the map $\mathcal{A} \to \mathcal{A}/(\mathrm{Ker}(x))$ which is 
bounded and hence induces a bounded homomorphism of Banach rings $\mathcal{A} \to \mathcal{H}(x)$. This in turn defines a 
morphism $\tilde{A} \to \widetilde{\mathcal{H}(x)}$. We have in fact defined a map 
$\mathcal{M}(\mathcal{A}) \to \mathrm{Spec}(\tilde{\mathcal{A}})$
which we call the reduction map and denote it by $\pi$. Explicitly stated,
\begin{align*}
    \pi &: \mathcal{M}(\mathcal{A}) \to \mathrm{Spec}(\tilde{\mathcal{A}}) \\
        &      x       \mapsto \mathrm{Ker}(\tilde{\mathcal{A}} \to \widetilde{\mathcal{H}(x)}).  
 \end{align*}
   
  For a $k$-affinoid space $X := \mathcal{M}(\mathcal{B})$ where $\mathcal{B}$ is a $k$-affinoid algebra, we will use 
$\tilde{X}$ to
 denote the 
$\tilde{k}$-scheme $\mathrm{Spec}(\tilde{\mathcal{B}})$. Concerning the reduction map in the case of 
strict $k$-affinoid algebras [\cite{berk}, 2.1] we have the following proposition.
\begin{prop} 

[\cite{berk}, 2.4.4] [\cite{berk3}, 2.3.6].
  Let $\mathcal{A}$ be a strict $k$-affinoid algebra. Set $X := \mathcal{M}(\mathcal{A})$,
   $\tilde{X} := \mathrm{Spec}(\tilde{\mathcal{A}})$,
 and 
  let $\tilde{X}_{gen}$ be the set of generic points of the irreducible components of the scheme $\tilde{X}$. 
With this notation, the following statements are true. 
\begin{enumerate}
 \item The reduction map $\pi : X \to \tilde{X}$ is surjective.
 \item For any $\tilde{x} \in \tilde{X}_{gen}$, there exists a unique point $x \in X$ with $\pi(x) = \tilde{x}$. If  
$\rho(\mathcal{A}) = |k|$ then there is an isomorphism 
$\tilde{k}(\tilde{x}) \simeq \widetilde{\mathcal{H}(x)}$.
\item The set $\pi^{-1}(\tilde{X}_{gen})$ is the Shilov boundary of $X$.
\item The pre-image of an open (resp. closed) subset of $\tilde{X}$ is closed (resp. open) under the morphism
 $\pi$.
\end{enumerate}
\end{prop}      

\subsubsection{Formal Covers}
    In this section we define
for any projective $k$-variety $V$, the $\tilde{k}$-scheme $\tilde{V}$ and the reduction map
$V^{\mathrm{an}} \to \tilde{V}$. One way of doing so  
would be to use an affinoid covering of the space $V^{\mathrm{an}}$ 
and glue the $\tilde{k}$-schemes arising from each element of the covering. 
However in order for the gluing to make sense, we must be restrictive in our choice of covering. It is to this end that we introduce the
notion of a formal cover of a separated $k$-analytic space.      

\begin{defi} [\cite{berk}, Section 4.3]
   \emph{An affinoid domain $W$ in a $k$-affinoid space $X$ is said to be} formal \emph{if the induced morphism 
$\tilde{W} \to \tilde{X}$
 is an open immersion.} 
\end{defi}

   The definition that follows is stated for a separated $k$-analytic space. Analytic spaces are constructed 
   and studied in detail in \cite{berk2}. 
The analytification of a $k$-variety is an example of a separated $k$-analytic space. We will be interested only in
this case. 

\begin{defi}
  \emph{Let $X$ be a separated $k$-analytic space. An affinoid covering $\{W_i\}$ is} formal \emph{if
the $W_i$ are strict $k$-affinoid spaces and
for any $W_i, W_j$ belonging to the cover, $W_i \cap W_j$ is a formal affinoid subdomain of both $W_i$ and $W_j$.} 
\end{defi}
 
 Let $X$ be a separated $k$-analytic space provided with a formal cover $\mathfrak{W} := \{W_i\}$.
 Gluing the $\tilde{W}_i$ defines a
$\tilde{k}$-scheme $\tilde{X}_\mathfrak{W}$ which is reduced and of finite type. Furthermore, the reduction
map $\pi : W_i \to \tilde{W}_i$ for each element of the covering extends 
to a map $\pi : X \to \tilde{X}_\mathfrak{W}$.

   Our notation $\tilde{X}_\mathfrak{W}$ was to specify the importance of the choice of cover 
involved in defining the scheme $\tilde{X}$. In what follows we will suppress the sub script and simply write
$\tilde{X}$. 
 
  The following proposition ensures that if $X$ is a projective $k$-variety of finite type then it always admits a formal 
cover. We first prove the theorem when $X = \mathbb{P}^{n,\mathrm{an}}_k$ for some $n$. The result in this case is obtained by 
exploiting the standard chart associated to $\mathbb{P}^{n}_k$. 

\begin{prop}\label{formal}
   Let $V$ be a projective $k$-variety.
Let $L/k$ be a complete non-Archimedean real valued field extension. 
 There exists a finite formal cover $\mathfrak{W}_L$ of 
$V^{\mathrm{an}}_L$ such that 
the collection $\{\mathfrak{W}_L\}_L$ has the following property. 
Let $L_1/k$ and $L_2/k$ be complete non-Archimedean real valued field extensions such that 
$L_1 \subset L_2$. If $\mathfrak{W}_{L_1} = \{D_i\}_i$ then $\mathfrak{W}_{L_2} = \{D_i {\times}_{L_1} L_2\}_i$.        
\end{prop}
\begin{proof}
 In what follows we will explicitly construct the formal cover $\mathfrak{W}_k$.
It can be verified that if $L/k$ is a complete
non-Archimedean real valued  field extension then extending scalars for each element of the cover $\mathfrak{W}_k$ by 
$L$ defines a formal cover $\mathfrak{W}_L$.

  We begin by considering the case when 
$V = \mathbb{P}^n_k$ for some $n \in \mathbb{N}$. The analytification of the projective space
 $\mathbb{P}^{n,\mathrm{an}}_k$
 can be described in a fashion reminiscent of the \textasciigrave Proj\textasciiacute construction in the
theory of schemes [\cite{Liu}, 2.3.3]. 

 Consider the $k$-algebra $k[T_1,\ldots,T_{n+1}]$. Let $S$ denote the set of all multiplicative
seminorms on this algebra which restrict to the valuation of the field $k$ such that 
if $x \in S$ then $|T_i(x)| \neq 0$ for some $i$. We define an equivalence relation $\sim$
 on $S$ as follows.
\begin{align*}
x \sim y \Longleftrightarrow  \mbox{ There exists } c \in \mathbb{R}_{>0} 
\mbox{ such that for any homogenous } \\ f \in k[T_1,..T_{n+1}],
 |f(x)| = c^{deg(f)}|f(y)| .
\end{align*}
    The set $S/\sim$ can be endowed with a topology in a natural fashion [\cite{baker}, 2.2] so that it becomes a compact, Hausdorff
topological space. We proceed further and give $S/\sim$ the structure of a $k$-analytic space. Let
\begin{align*}
 A_j := \{ x \in S | |T_r(x)| \leq |T_j(x)| \mbox{ for every } 1 ≤ r ≤ n+1 \}. 
\end{align*}
   Observe that if $a \in A_j$ and $a \sim b$ then $b \in A_j$. We will
abuse notation and denote $A_j/\sim$ by $A_j$ as well. It follows that $S/\sim = \cup_j A_j$.
 Furthermore, for any $1 ≤ j ≤ n+1$,
$A_j$ is in bijection with the set of multiplicative seminorms on 
$k[T_1/T_j,.,T_{n+1}/T_j]$ which when evaluated at $(T_i/T_j)$ for any possible choice of $i$
 is less than or equal to $1$. But
 this is exactly the set of all multiplicative
seminorms on the affinoid algebra $\mathcal{B}_j := k\{T_1/T_j,.,T_{n+1}/T_j\}$ which restrict to the given valuation on $k$. In fact we have a 
homeomorphism
\begin{align*}
 \delta_j : A_j \to \mathcal{M}(\mathcal{B}_j).
\end{align*}
    Consequently, the collection $\{ \cap_{r \in Q} A_{r} \}_{Q \in \mathfrak{P}}$ where $\mathfrak{P}$ is the set of all subsets of the 
    set $\{1,\ldots,n+1\}$, forms a
net [\cite{berk2}, 1.1] of compact sets on the space $S/\sim$.
 If $i \neq j$ then
 $\delta_j$ restricts to an isomorphism between 
  $A_j \cap A_i$ and the affinoid domain
\begin{align*}
 \mathcal{M}(\mathcal{B}_{ji}) := k\{T_1/T_j,..T_i/T_j,(T_i/T_j)^{-1},.,T_{n+1}/T_j\}.  
\end{align*}
   We use $\delta_{ij}$ to denote this isomorphism. Note that 
$\mathcal{B}_{ji} = \mathcal{B}_{ij}$ and the isomorphisms $\delta_{ij}$ and $\delta_{ji}$ are the same. 

    Similarly let $Q \in \mathfrak{P}$. The space $\cap_{r \in Q \cup \{i\}} A_r$ is homeomorphic 
    to the affinoid space 
 \begin{align*}
   \mathcal{M}(\mathcal{B}_{iQ}) := \mathcal{M}(k\{T_j/T_i,(T_s/T_i)^{-1}\})_{s \in Q, j \in \{1,..,n+1\}}     
 \end{align*}
    via the restriction of any $\delta_s$ where $s \in Q \cup i$. For any 
    $s   \in Q \cup i$, $\mathcal{B}_{iQ} = \mathcal{B}_{sQ}$ and the restrictions 
    $\delta_s$ are the same for all such $s$.   

   The triple 
$(S/\sim,(\mathcal{B}_{iQ})_{i \in {1,..,n+1},Q \in \mathfrak{P}},\cap_{r \in Q \cup i}A_r)$ is hence a 
 $k$-analytic space [\cite{berk2}, pg $17$] and is isomorphic
as an analytic space to $\mathbb{P}^{n,\mathrm{an}}_k$.
 
This description of $\mathbb{P}^{n,\mathrm{an}}_k$ enables us to see it as the union of $n+1$ affinoid domains, each isomorphic
to the $n$-dimensional Berkovich closed polydisc of poly radius $(1,\ldots,1)$.     
   We claim that $\{A_j\}$ is a formal cover of $\mathbb{P}^{n,\mathrm{an}}_k$. 

Since
\begin{align*}
\widetilde{\mathcal{B}_{ji}} = \tilde{k}[T_1/T_j,..T_i/T_j,(T_i/T_j)^{-1},..T_{n+1}/T_j], 
\end{align*} 
the $\tilde{k}$-algebra
  $\widetilde{\mathcal{B}_{ji}}$ corresponds to an
 open affine sub scheme 
of $\widetilde{\mathcal{M}({\mathcal{B}}_j)}$ and $\widetilde{A_j \cap A_i}$ is 
an open affine subset of $\tilde{A}_j$.
 We conclude that our claim is verified, thus proving 
the proposition for the case $V = \mathbb{P}^{n,\mathrm{an}}_k$.   
  
  In the general case we make use of the fact that $V$ is projective and hence can be seen as a closed subset of 
$\mathbb{P}^n_k$ for some $n$. Furthermore for every $j$, $\delta_j$ restricts to an isomorphism
\begin{align*}
 \delta_j : V^{\mathrm{an}} \cap A_j \to \mathcal{M}(k\{T_1/T_j,..T_i/T_j,..T_{n+1}/T_j\}/I_j)
\end{align*}
 where $I_j$ is an ideal contained in the polynomial ring
 $k[T_1/T_j,..T_{n+1}/T_j]$ and determined by the embedding of $V$ into $\mathbb{P}^n_k$. 
 The generators of $I_j$ can be chosen to be polynomials $\{f_1,\ldots,f_u\}$ belonging to 
 $k[T_1/T_j,..T_{n+1}/T_j]$ such that if 
 $\rho$ denotes the Gauss norm (or spectral norm) of the affinoid algebra 
 $k\{T_1/T_j,..T_i/T_j,.,T_{n+1}/T_j\}$ then $|f_i(\rho)| \leq 1$. The justification for this follows
 from [\cite{berk}, (2) Page 64].

   We claim that the cover $\{V^{\mathrm{an}} \cap A_i\}_i$ is a formal cover of $V^{\mathrm{an}}$. 
 Let $1 \leq i,j \leq n+1$ with $i \neq j$. We need only show that  
the affinoid space $V^{\mathrm{an}} \cap A_i \cap A_j$ is a formal domain in 
$V^{\mathrm{an}} \cap A_i$. We may assume that $V^{\mathrm{an}} \cap A_i \cap A_j$
is not empty, since otherwise it is trivially a formal domain. 
Observe that we have the following equality. 
\begin{align*}
(V^{\mathrm{an}} \cap A_i) \cap A_j = \{x \in V^{\mathrm{an}} \cap A_i | |T_j/T_i(x)|  = 1\}. 
\end{align*}
By assumption, there exists an $x \in V^{\mathrm{an}} \cap A_i$ such that 
$|T_j/T_i(x)|  = 1$. It follows that if $\rho$ denotes the spectral norm of 
the strict affinoid algebra corresponding to $V^{\mathrm{an}} \cap A_i$ then $|T_j/T_i(\rho)|  = 1$. 
The above equality and [\cite{berk}, (ii) pg.28] imply that the  
affinoid algebra corresponding to the space $(V^{\mathrm{an}} \cap A_i) \cap A_j$ 
is $(\mathcal{B}_i/I_i)\{T_i/T_j\}$. Applying 
 [\cite{BGR}, Proposition 7.2.6/3] will give the result. 
    
    For a complete non-Archimedean real valued field extension $L/k$, it can be checked that 
$\mathfrak{W}_L := \{V^{\mathrm{an}}_L \cap (A_j {\times}_k L)\}_j$ forms a formal cover of 
$V^{\mathrm{an}}_L$.
  \end{proof}

   The topological subspace $A_j \subset \mathbb{P}^{n,\mathrm{an}}_k$ is endowed with the structure of a 
$k$-affinoid space via the 
homeomorphism $\delta_j$. Hence we will refer to it in future as an affinoid domain in $\mathbb{P}^{n,\mathrm{an}}_k$
and identify it with the strict affinoid space $\mathcal{M}(\mathcal{B}_j)$.

\subsection{ The class of open sets $\mathcal{O}_L$ }   
Let $L/k$ be a complete non-Archimedean real valued algebraically closed field extension. Let
$x \in \mathbb{P}_k^{n,\mathrm{an}}(L)$ be an $L$-point of the analytification of projective $n$-space.
As outlined in Remark 1.1,
the pair $x : \mathrm{Spec}(L) \to \mathbb{P}_k^{n}$, 
$\mathrm{id}_L : \mathrm{Spec}(L) \to \mathrm{Spec}(L)$ defines a closed point of the 
variety $V_L = V \times_k \mathrm{Spec}(L)$ which we denote $x_L$. 
We proceed to define the family $\mathcal{O}_{x_L,\mathbb{P}_L^{n,\mathrm{an}}}$ of open neighbourhoods of $x_L$ 
    
 In Proposition 2.6 we showed that having chosen an affine chart of $\mathbb{P}^n_k$, the space $\mathbb{P}^{n,\mathrm{an}}_k$ can be seen to be the union of $n+1$ 
$n$-dimensional
Berkovich closed disks defined over $k$. We denoted this collection $\{A_i\}_i$. 
Likewise $\mathbb{P}^{n,\mathrm{an}}_L = \bigcup_i A_{i,L}$ where
$\{A_{i,L} := A_i {\times}_k L\}_i$ forms a collection of    
 $n+1$ 
$n$-dimensional
Berkovich closed disks defined over $L$. For some $j$, we must have that $x_L \in A_{j,L}$. 
Let $\mathcal{O}_{x_L,\mathbb{P}_L^{n,\mathrm{an}}}$ be the family
of Berkovich open balls containing $x_L$ and contained in $A_{j,L}$.
 Each such Berkovich open ball centered at $x_L$ will be in bijection with an $n$-tuple 
of positive real numbers less than or equal to $1$.
 We proceed below in greater detail. 
 
 In the 
proof of Proposition $2.6$, we introduced the following notation concerning the
$A_{j,L}$. 
\begin{align*}
   A_{j,L} = \mathcal{M}({\mathcal{B}}_{j,L})
\end{align*}
where
\begin{align*}
 \mathcal{B}_j := k\{T_1/T_j,..T_i/T_j,..T_{n+1}/T_j\}
\end{align*}
  and ${\mathcal{B}}_{j,L} := \mathcal{B}_j {\times}_k L$.
  The affinoid space $\mathcal{M}(\mathcal{B}_{j,L})$ is
 an $n$-dimensional Berkovich closed disk over $L$.
The point $x_L \in \mathbb{P}^n_L$ is a closed point defined over $L$. Let it
 have coordinates $[x_{1,L}:\ldots:x_{{n+1},L}]$ and $i$ be any index such that 
 $|x_{j,L}| \leq |x_{i,L}|$ for every $j \in \{1,..,n+1\}$. By definition of the space $A_{i,L}$, $x_L$ must belong to it.
Using the fact that $A_{i,L} = \mathcal{M}(\mathcal{B}_{i,L})$ we 
define a family of open neighbourhoods of $x_L$, namely the collection of Berkovich open balls defined
by the equations $|(T_j/T_i - x_{j,L}/x_{i,L})(p)| < r_j$, where $j \neq i$ and $r_j \leq 1$.
 If $x_L \in A_{t,L}$ then $|x_{i,L}| = |x_{t,L}|$
  and any such Berkovich open sub ball of $A_{i,L}$ will also be contained in 
$A_{t,L}$.
By this we mean that there exists $\mathbf{r}_t := (r'_1,\ldots,r'_{n+1}) \in (0,1]^{n+1}$ such that
$r'_t = 1$ and $B = \{p \in A_{t,L}| \wedge_{j \neq t} [|(T_j/T_t - x_{j,L}/x_{t,L})(p)| < r'_j] \}$. 
It can be shown that as $B$ 
varies through all the Berkovich open sub balls of $A_{i,L}$ which contain $x_L$, it also varies through      
all Berkovich open sub balls of $A_{s,L}$ which contain $x_L$ i.e. for any $s$ such that $|x_{s,L}| = |x_{i,L}|$. Hence, we may define  
the family $\mathcal{O}_{x_L,\mathbb{P}_L^{n,\mathrm{an}}}$ to be the collection $\{B(x_L,\mathbf{r}_i)\} \subset A_{i,L}$
 for any $i$ such that $|x_{j,L}| \leq |x_{i,L}|$ for every 
$j \in \{1,..,n+1\}$ 
and where 
$\mathbf{r}_i := (r_1,\ldots,r_i,\ldots,r_{n+1})$ is any $n +1$-tuple for which $r_i =1$ and
$r_h \leq 1$ for any other $h$.
Let $\mathcal{O}_{L,\mathbb{P}_L^{n,\mathrm{an}}} := \bigcup_{x \in \mathbb{P}^{n,\mathrm{an}}_k(L)} \mathcal{O}_{x_L,\mathbb{P}_L^{n,\mathrm{an}}}$. 

To an
element $W$ of this family we associate an $(n+1)^2$-tuple. That is, we define a function 
$h_{L,\mathbb{P}_L^{n,\mathrm{an}}} : \mathcal{O}_{L,\mathbb{P}_L^{n,\mathrm{an}}} \to \mathbb{R}^{(n+1)^2}$. If for an index $t$, $|x_{t,L}| = |x_{i,L}|$ then let 
$\mathbf{r}_t = (r_1,\ldots,r_{n+1})$ be such that $r_t = 1$ and the Berkovich open ball 
$W$ is defined by the equations $|(T_j/T_t - x_{j,L}/x_{t,L})(p)| < r_j$ for $j \neq t$.
 If on the other hand $|x_{t,L}| < |x_{i,L}|$
 for some $i \in \{1,\ldots,n+1\}$
  then let $\mathbf{r}_t = (1,\ldots,1)$. 
Let $h_{L,\mathbb{P}_L^{n,\mathrm{an}}}(W) := (\mathbf{r}_i)_i$.

     If $V$ is an arbitrary projective $k$-variety then by definition it admits an embedding $V \hookrightarrow \mathbb{P}^{n}_k$. 
Let $x \in V^{\mathrm{an}}(L) \subset \mathbb{P}^{n,\mathrm{an}}_k(L)$. 
We defined a family of open neighbourhoods of $x_L$ in $\mathbb{P}^{n,\mathrm{an}}_L$ which we
called
 $\mathcal{O}_{x_L,\mathbb{P}_L^{n,\mathrm{an}}}$.
 Along with this family of open neighbourhoods, we also defined a function 
 $h_{L,\mathbb{P}_L^{n,\mathrm{an}}}$ which defines the poly radius of every element in $\mathcal{O}_{x_L,\mathbb{P}_L^{n,\mathrm{an}}}$.
  Restricting 
every element of the family $\mathcal{O}_{x_L,\mathbb{P}^{n,\mathrm{an}}_L}$
  to $V_L^{\mathrm{an}}$ will define a family $\mathcal{O}_{x_L, V_L^{\mathrm{an}}}$ of open 
neighbourhoods of $x_L$ in $V^{\mathrm{an}}_L$. Hence 
$\mathcal{O}_{x_L,V_L^{\mathrm{an}}} = \{W \cap V_L^{\mathrm{an}} | W \in \mathcal{O}_{x_L,\mathbb{P}_L^{n,\mathrm{an}}}\}$.  
Let $W \in \mathcal{O}_{x_L,V_L^{\mathrm{an}}}$. We define $Q_W$ to be the collection of 
$W' \in \mathcal{O}_{x_L,\mathbb{P}^{n,\mathrm{an}}_L}$ such that 
$W' \cap V^{\mathrm{an}}_L = W$. We set 
$h_{L,V_L^{\mathrm{an}}}(W \cap V_L^{\mathrm{an}}) := \mathrm{inf}_{W' \in Q_W} \{h_{L,\mathbb{P}_L^{n,\mathrm{an}}}(W')\}$. 
The infimum here is taken with respect to 
the coordinate wise partial ordering defined on $\mathbb{R}^{{n+1}^2}$ 
i.e. $(x_n)_n \leq (y_n)_n$ if and only if $x_i \leq y_i$ for every $1 \leq i \leq (n+1)^2$. 
By Lemma 2.7, the function $h_{L,V_L^{\mathrm{an}}}$ is well defined. 

\begin{lem}
  Let $x_L \in V^{\mathrm{an}}_L(L)$ and $W \in \mathcal{O}_{x_L, V_L^{\mathrm{an}}}$. 
The set of $(n+1)^2$-tuples 
 $\{h_{L,\mathbb{P}_L^{n,\mathrm{an}}}(W') | W' \in Q_W\} \subset \mathbb{R}^{(n+1)^2}$
has a well defined infimum with respect to the coordinate wise partial ordering defined on $\mathbb{R}^{(n+1)^2}$. 
 \end{lem} 
(Note : This infimum need not belong to the set $\{h_{L,\mathbb{P}^{n,\mathrm{an}}_L}(W') | W' \in Q_W\}$). 
\begin{proof}
Let $P := \{h_{L,\mathbb{P}_L^{n,\mathrm{an}}}(W') | W' \in Q_W\}$. 
   We are required to show that there exists $(s_i)_i \in \mathbb{R}_{\geq 0}^{(n+1)^2}$ such that 
\begin{enumerate}
 \item If $(y_i)_i \in P$ then $s_j \leq y_j$ for every $j$. 
\item Let $(z_i)_i \in \mathbb{R}^{(n+1)^2}$ be such that for any  
$(y_i)_i \in P$ the inequality 
         $z_j \leq y_j$ holds for every $j$. Then we must have $z_j \leq s_j$ for every $j$. 
          \end{enumerate}
          
     Let $p_i$ denote the $i$-th projection morphism $\mathbb{R}^{(n+1)^2} \to \mathbb{R}$.
      Let $s_i \in \mathbb{R}_{\geq 0}$ denote that real number 
 which is the infimum of the set $\{p_i(x) | x \in P\}$. The 
 $(n+1)^2$ tuple $(s_i)_i$ can be checked to satisfy property (1). 
 
     To show that $(s_i)_i$ satisfies property (2) as well, we exhibit a sequence of elements in $P$ 
 which converges to $(s_i)_i$. Let $j \in\{1,\ldots,(n+1)^2\}$. By definition, there
 exists a  sequence of open neighbourhoods $(W_{j,m})_m$ in $\mathcal{O}_{x_L,\mathbb{P}_L^{n,\mathrm{an}}}$
such that $(p_j \circ h_{L,\mathbb{P}^{n,\mathrm{an}}_L}(W_{j,m}))_m$ converges to $s_j$. 
Let $W_m := \bigcap_j W_{j,m}$. The set $W_m$ is an open neighbourhood belonging to 
$Q_W$. 
Also, $(p_j \circ h_{L,\mathbb{P}_L^{n,\mathrm{an}}})(W_{m}) \leq (p_j \circ h_{L,\mathbb{P}_L^{n,\mathrm{an}}})(W_{j,m})$ for 
every $j$. 
It follows from the construction of $(s_i)_i$ that for any 
$j \in \{1,\ldots,(n+1)^2\}$, the sequence $((p_j \circ h_{L,\mathbb{P}_L^{n,\mathrm{an}}})(W_{m}))_m$ converges
to $s_j$. This is equivalent to saying that $(h_{L,\mathbb{P}_L^{n,\mathrm{an}}}(W_{m}))_m$ converges to 
$(s_i)_i$. 
  \end{proof}       
   
   In Remark 1.5, we introduced a family of functions 
$\mathbb{R}_{\geq 0}^{(n+1)^2} \to \mathbb{R}$ which we denoted $S$.

\begin{lem}
  Let $L/k$ be an algebraically closed complete non-Archimedean real valued field extension and 
$g \in S$. 
  Let $\mathcal{O}_1$ and $\mathcal{O}_2$
 belong to $\mathcal{O}_{x_L,V_L^{\mathrm{an}}}$ such that $\mathcal{O}_1 \subset \mathcal{O}_2$.
The following inequality holds true. 
\begin{align*}
 (g \circ h_{L,V_L^{\mathrm{an}}})(\mathcal{O}_1) \leq (g \circ  h_{L,V_L^{\mathrm{an}}})(\mathcal{O}_2). 
\end{align*}
\end{lem}
\begin{proof}
  A straight forward argument reduces to proving the lemma when $V = \mathbb{P}^{n}_k$.  
  In the course of the proof we will make use of this fact: Since the field $L$ 
  is algebraically closed and endowed with a non-trivial valuation, its 
value group is dense in $\mathbb{R}_{\geq 0}$. 
Let $h_{L,\mathbb{P}_L^{n,\mathrm{an}}}(\mathcal{O}_1) := (\mathbf{r}_1,\ldots,\mathbf{r}_{n+1})$
and $ h_{L,\mathbb{P}_L^{n,\mathrm{an}}}(\mathcal{O}_2) := (\mathbf{r}'_1,\ldots,\mathbf{r}'_{n+1})$.  
If $\mathbf{r}_t = (r_{1,t},\ldots,r_{n+1,t})$ and $\mathbf{r}'_t = (r'_{1,t},\ldots,r'_{n+1,t})$ then 
we claim that for every $i,t$, $r_{i,t} \leq r'_{i,t}$. We proceed by assuming the contrary.
 If for some $t$ there exists an $i$ such that $r_{i,t} > r'_{i,t}$ then 
we can find an element $y_i \in L$ 
such that $r'_{i,t} < |(y_i/x_t - x_i/x_t)| < r_{i,t}$. The element $[x_1:..:y_i:..:x_{n+1}]$ will belong 
to $\mathcal{O}_1$ but not $\mathcal{O}_2$. This is not possible and we must hence have $r_i \leq r'_i$. 
Our choice of the function $g$ implies that the inequality 
 $(g \circ  h_{L,\mathbb{P}_L^{n,\mathrm{an}}})(\mathcal{O}_1) \leq (g \circ  h_{L,\mathbb{P}_L^{n,\mathrm{an}}})(\mathcal{O}_2)$ holds. 
\end{proof}

 The following lemma implies that the family of open neighbourhoods $\mathcal{O}_{x_L,V_L^{\mathrm{an}}}$ for 
$x \in V^{\mathrm{an}}(L)$ does not depend on the extension $L/k$. 
\begin{lem}
  Let $L'/k$ and $L/k$ be 
complete non-Archimedean real valued algebraically closed field extensions
such that $L \subseteq L'$. Let $x \in V^{\mathrm{an}}(L)$. We have the following equality of sets.
\begin{align*}
   \mathcal{O}_{x_{L'},V_{L'}^{\mathrm{an}}} = \{O {\times}_L L'|O \in \mathcal{O}_{x_L,V_L^{\mathrm{an}}}\}.    
\end{align*}
   Furthermore, if $O \in \mathcal{O}_{x_L,V_L^{\mathrm{an}}}$ then $h_{L',V_{L'}^{\mathrm{an}}}(O {\times}_L L') =  h_{L,V_L^{\mathrm{an}}}(O)$. 
 \end{lem}
\begin{proof}
    It can be inferred from the discussion above concerning the family $\mathcal{O}_{x_L,V_L^{\mathrm{an}}}$ 
that it suffices to prove the lemma assuming $V = \mathbb{P}^{n}_k$ for some $n \in \mathbb{N}$. 
If $x \in \mathbb{P}^{n,\mathrm{an}}_k(L)$ then there exists exactly one point on the fibre 
over $x_L$ for the projection morphism $\mathbb{P}^n_L \times_L L' \to \mathbb{P}^n_L$ and 
$(x_L)_{L'} = x_{L'}$. We assume without loss of generality that $x_L \in A_{i,L}$ with homogenous
coordinates $[x_{1,L}:\ldots:x_{n+1,L}]$. If $x_{L'}$ has homogenous coordinates    
 $[x_{1,L'}:\ldots:x_{n+1,L'}]$ then $x_{j,L}/x_{i,L} = x_{j,L'}/x_{i,L'}$ for all $j$. 

    Let $O \in \mathcal{O}_{x_L,\mathbb{P}_L^{n,\mathrm{an}}}$. By definition, $O$ must be of the form 
$\{p \in A_{i,L}| \wedge_{j} |(T_j/T_i - x_{j,L}/x_{i,L})(p)| < r_j \}$. It follows that     
$O {\times}_L L' = \{p \in A_{i,L'}| \wedge_{j} |(T_j/T_i - x_{j,L'}/x_{i,L'})(p)| < r_j \}$ which 
is an element of $\mathcal{O}_{x_{L'},\mathbb{P}_{L'}^{n,\mathrm{an}}}$. Hence 
$\{O {\times}_L L'|O \in \mathcal{O}_{x_L,\mathbb{P}^{n,\mathrm{an}}}\} \subset \mathcal{O}_{x_{L'},\mathbb{P}_{L'}^{n,\mathrm{an}}}$ and 
$h_{L',\mathbb{P}_{L'}^{n,\mathrm{an}}}(O {\times}_L L') = h_{L,\mathbb{P}_L^{n,\mathrm{an}}}(O)$.

    Let $O \in \mathcal{O}_{x_{L'},\mathbb{P}_{L'}^{n,\mathrm{an}}}$. By definition, $O$ must be of the form 
$\{p \in A_{i,L'}| \wedge_{j} |(T_j/T_i - x_{j,L'}/x_{i,L'})(p)| < r_j \}$.
Using the equality $x_{j,L}/x_{i,L} = x_{j,L'}/x_{i,L'}$, 
the image of this 
open set under the projection morphism $A_{i,L} {\times}_L L' \to A_{i,L}$
is of the form $O_0 := \{p \in A_{i,L}| \wedge_{j} |(T_j/T_i - x_{j,L}/x_{i,L})(p)| < r_j \}$. 
It follows that $O_0 {\times}_L L' = O$. Hence 
$\mathcal{O}_{x_{L'},\mathbb{P}_{L'}^{n,\mathrm{an}}} = \{O {\times}_L L'|O \in \mathcal{O}_{x_L,\mathbb{P}_L^{n,\mathrm{an}}}\}$.         
\end{proof}

          As Theorem 1.7 concerns the function $h_{L,\mathbb{P}_L^{n,\mathrm{an}}}$ and the family $\mathcal{O}_{x_L,\mathbb{P}_L^{n,\mathrm{an}}}$ rather than 
          $h_{L,V_L^{\mathrm{an}}}$ and $\mathcal{O}_{x_L,V^{\mathrm{an}}}$, we will henceforth make no reference to $h_{L,V_L^{\mathrm{an}}}$ and 
         $\mathcal{O}_{x_L,V^\mathrm{an}}$. We simplify notation and use $\mathcal{O}_{x_L}$ to denote  $\mathcal{O}_{x_L,\mathbb{P}_L^{n,\mathrm{an}}}$, 
         $h_L$ in place of $h_{L,\mathbb{P}_L^{n,\mathrm{an}}}$ and $\mathcal{O}_L$ for $\mathcal{O}_{L,\mathbb{P}_L^{n,\mathrm{an}}}$.

\subsubsection{ The family $\mathcal{O}_{x_L}$ when $x_L \in \mathbb{P}_L^{1,\mathrm{an}}$. }
   
             The space $\mathbb{P}^{1,\mathrm{an}}_L$ admits two descriptions, one of which 
     was outlined in the proof of Proposition 2.6.  
    We briefly describe these constructions.  
    
     Let $S$ denote the set of all multiplicative
seminorms on the polynomial algebra
$L[T_1,T_2]$ 
 which restrict to the valuation on the field $L$ such that 
if $x \in S$ then it cannot be that $|T_1(x)| = |T_2(x)| = 0$. We define an equivalence relation $\sim$
 on $S$ as follows.
\begin{align*}
x \sim y \Longleftrightarrow  \mbox{ There exists } c \in \mathbb{R}_{>0} 
\mbox{ such that for any homogenous } \\ f \in k[T_1,T_2],
 |f(x)| = c^{deg(f)}|f(y)| .
\end{align*}
   Let
\begin{align*}
 A_2 := \{ x \in S | |T_1(x)| \leq |T_2(x)| \}
\end{align*}
    and 
\begin{align*}
 A_1 := \{ x \in S | |T_2(x)| \leq |T_1(x)|\}. 
\end{align*}     
     The subspaces $A_1$ and $A_2$ are stable under the equivalence relation. We will abuse notation and refer 
     to their images in $S/\sim$ as $A_1$ and $A_2$ as well.  
     
   One can also describe the space $\mathbb{P}^{1,\mathrm{an}}_L$ in the following manner.
  Following Example $2.2$, 
   the space $\mathbb{A}^{1,\mathrm{an}}_L$ can be realised as the set of multiplicative 
   seminorms on the algebra $L[T]$ which restrict to the norm on $L$. We endow this set 
   with the weakest topology such 
   that if $f \in L[T]$ then the function from $\mathbb{A}^{1,\mathrm{an}}_L$ to $\mathbb{R}$ defined by 
   $x \mapsto |f(x)|$ is continuous. Let $y \in \mathbb{A}^{1,\mathrm{an}}_L(L)$. This means that 
   $y$ corresponds to a morphism $L[T] \to L$. Such a morphism defines a seminorm on $L[T]$ and 
   hence corresponds to a point on $\mathbb{A}^{1,\mathrm{an}}_L$ . 
    With this topology the sets $B(y,r) := \{p \in \mathbb{A}^{1,\mathrm{an}}_L| |(T-y)(p)| < r\}$ form a family of 
    open neighbourhoods around 
    the point $y$. These open sets are precisely the Berkovich open balls around $y$ of radius $r$. 
 As a set $\mathbb{P}^{1,\mathrm{an}}_L =   \mathbb{A}^{1,\mathrm{an}}_L \cup \infty$.  A basis of 
 open neighbourhoods around the 
 point $\infty$ is given by sets of the form $\{p \in \mathbb{P}^{1,\mathrm{an}}_L| |T(p)| > r\}$.
 
    We now identify these two descriptions of $ \mathbb{P}^{1,\mathrm{an}}_L$ in order to relate the 
    family $\mathcal{O}_{x_L}$ as $x_L$ varies along the $L$-points of $\mathbb{A}^{1,\mathrm{an}}_L$.
    We go back to the first description of $ \mathbb{P}^{1,\mathrm{an}}_L$. Let $S' \subset S$ 
    denote the sub collection of seminorms 
    such that if $p \in S'$ then $|T_2(p)| \neq 0$. The set $S'$ is stable for the equivalence relation and set $A' := S'/\sim$. 
    By definition, the elements of $A'$ define multiplicative seminorms on the algebra $L[T_1/T_2]$. 
    Hence we have a function from $A' \to  \mathbb{A}^{1,\mathrm{an}}_L$. 
    With the induced topology on $A'$ it can be shown that we have a homeomorphism 
    $H : A' \to \mathbb{A}^{1,\mathrm{an}}_L$. 
    
    An $L$-point of $S'/\sim$ can be uniquely described by means 
    of homogenous coordinates as in the Proj construction. Let $x_L \in S'/\sim$ then 
    $x_L$ can be represented by homogenous coordinates $[a:b]$ where 
    $a, b \in L$ and
    $b \neq 0$. By definition, $H([a:b]) = a/b$, where by $a/b$ we mean the $L$-point on 
    $\mathbb{A}^{1,\mathrm{an}}_L$ defined by the equation $T_1/T_2 = a/b$.
    For the calculations that follow, we will assume without loss of generality that 
    $x_L = [a:1]$.
    
     Let  $B(H(x_L),r)$ be a Berkovich open ball
    around the point $H(x_L)$ of radius $r \leq 1$. By definition $B(H(x_L),r) \in \mathcal{O}_{x_L}$. We will now
    write down its associated 4-tuple $h_L(B(H(x_L),r))$. We divide the problem into three cases. \\

   (1). Let $|a| < 1$. The point $x_L$ does not belong to the closed subspace $A_1$. By definition of the function $h_L$   
     we have that $h_L(B(H(x_L),r)) = ((1,1),(r,1))$.  \\
     
(2). Let $|a| = 1$. The point $x_L$ belongs to both $A_1$ and $A_2$. 
      We  then have that $h_L(B(H(x_L),r)) = ((1,r),(r,1))$. \\ 
      
(3). Let $|a| > 1$. The point $x_L$ does not belong to the closed space $A_2$.  
     It can be shown that $h_L(B(H(x_L),r)) =  ((1,r/|a|^2),(1,1))$.  \\
     
     Similarly, every element of the family $\mathcal{O}_{x_L}$ corresponds to a Berkovich open ball around $H(x_L)$. 
Let $O \in \mathcal{O}_{x_L}$ and let $B(H(x_L),s)$ be the corresponding Berkovich open ball around $x_L$ of radius $s$.  
The radius of this ball can be expressed in terms of the $4$-tuple, $h_L(O)$ as follows. \\
    
 (1). Let $|a| < 1$. If $h_L(O) = ((1,1),(r,1))$ then the corresponding Berkovich open ball around $x_L$ has radius $s = r$. \\
 
 (2). Let $|a| = 1$. If $h_L(O) = ((1,r),(r,1))$ then $s = r$. \\
 
 (3).  Let $|a| > 1$ and $h_L(O) = ((1,r),(1,1))$. If $r \leq 1/|a|$ then $s = r|a|^2$. If 
         $r > 1/|a|$ then $O$ is an open neighbourhood of the point $\infty$. It is a Berkovich open ball around 
         $\infty$ but is not a Berkovich open ball contained in $\mathbb{A}^{1,\mathrm{an}}_L$.  
  \\   
     
      The calculations above will be made use of in Section 6. In future, we will not refer to the 
      homeomorphism $H$ and use $x_L$ itself to denote $H(x_L)$.      

\subsection{Model theory}
    In this section we discuss a theorem of Hrushovski and Loeser which is integral to the proof of 
Theorem $1.7$. The theorem which was proved in \cite{loes} involves considerable use of Model theory \cite{Pillay}. 
In the following section we
 will make clear how the theorem implies a deep result about certain Berkovich analytic spaces.
The first chapter of
the paper \cite{loes} 
is devoted to introducing the basic notions of Model theory (Language, Theory, Model, Structure, type, definable type....)
 which we will assume the reader is familiar with and proceed to outlining 
the general framework within which the theorem is stated. In this section, given a 
non-Archimedean valued field $M$ we write its value group
$\Gamma_{\infty}(M)$ additively. 

    We use the many sorted language $L_\emph{G}$ consisting of the sorts $VF$, $\Gamma$ and 
$k$ for the valued field, the value group and the residue field with the ring, additive abelian group and residue field language respectively
as well as the geometric sorts $S_n$ and $T_n$ for $n \geq 1$ [\cite{loes}, 2.6]. The theory ACVF of algebraically closed valued fields 
 admits both elimination of imaginaries
and elimination of quantifiers in the language $L_\emph{G}$. 
In addition, we include a $0$ - definable point $\infty$ to the value group sort which denotes the valuation of $0 \in M$ for any model $M$ of ACVF. Let 
$\Gamma_{\infty}$ denote the sort $\Gamma \cup \infty$. The ordering on $\Gamma(M)$ induces a natural ordering on $\Gamma_{\infty}(M)$.         
 
   To make sense of Theorem $2.11$ one must understand the concept of 
a stably dominated type [\cite{loes}, Definition $2.5.2$]. 
 Let $A$ be some substructure of $\mathbb{U}$ where $\mathbb{U}$ denotes a large saturated model of ACVF. In the 
theory ACVF, 
 an 
$A$-definable type is stably dominated if and only if
 it is $\Gamma$-orthogonal [\cite{loes}, Proposition 2.8.1]. We now define a
$\Gamma$-orthogonal type.   

\begin{defi}
   \emph{ Let $A$ be a substructure of $\mathbb{U}$ and $p$ be an $A$-definable type. Then $p$ is said to be} $\Gamma$-orthogonal \emph{if
for any model $M$ containing $A$ and a realization $a$ of $p_{|M}$, we have that $\Gamma(\mathrm{dcl}(M \cup a)) = \Gamma(M)$ where
$\mathrm{dcl}(M \cup a)$ denotes the definable closure of the set $M \cup \{a\}$.} 
\end{defi}
    
   Let $V$ be a quasi-projective variety over a non-Archimedean valued field. Let
$X \subset V \times \Gamma_\infty^l$ be an $A$-definable set where
$A \subset VF \cup \Gamma_\infty$ and $l \in \mathbb{N}$.  
For any substructure $C$ containing $A$, we define $\widehat{X}(C)$
 to be the set of $C$-definable stably dominated types which concentrate on $X$ [\cite{loes}, 3.2].
 The space $\widehat{X}$ is a pro - definable set [\cite{loes}, 2.2]. 
  We 
briefly explain how the set $\widehat{X}$ can be given a topology.
 
   As a set $\widehat{V \times \Gamma_\infty^l} = \widehat{V} \times \Gamma_\infty^l$ [\cite{loes}, 3.4]. We endow $\widehat{V}$ with a topology
by defining a collection of pre-basic open sets. Any basic open set is the intersection of a finite number of pre-basic open sets. 
A pre-basic open set of $\widehat{V}$ is of the form 
$\{p \in \widehat{O}|\mathrm{val}(f)_*(p) \in W\}$ where $O \subset V$ is an open subspace of $V$, $f$ a regular function on $O$ and 
$W$ an open subset of $\Gamma_\infty$. The notation $\mathrm{val}(f)_*(p)$ requires some explanation. Firstly $\mathrm{val}(f)$ denotes the composition
$O \to VF \to \Gamma_\infty$. As $f$ is regular, $\mathrm{val}(f)$ is definable. $\mathrm{val}(f)_*(p)$ is a definable type in $\Gamma_\infty$ defined by
$d_{\mathrm{val}(f)_*(p)}(\phi(z,y)) = d_p(\phi(\mathrm{val}(f)(x),y))$. It is in fact a stably dominated type on $\Gamma_\infty$ and is hence constant
[\cite{loes}, 2.7.1]. The set $\widehat{V} \times \Gamma_\infty^l$ can be given the product topology and we let $X$ have the subspace topology.         
Like the space $V^{\mathrm{an}}$, the $k$-points of the variety $V(k)$ can be embedded into $\widehat{V}(k)$ which is the set of stably dominated 
$k$-definable types which concentrate on $V$ if we give $\widehat{V}(k) \subset \widehat{V}$ the topology generated 
 by the pro $k$ - definable open sets of $\widehat{V}$. 
 
    We are now in a position to state the Theorem of Hrushovski and Loeser [\cite{loes}, Theorem 10.1.1] which we will use later. 
     
\begin{thm}
  Let $V$ be a quasi projective variety over a non-Archimedean valued field. Let $A \subset VF \cup \Gamma_\infty$ and 
$X \subset V \times \Gamma_\infty^l$ be an $A$-definable set.
Then there exists an pro $A$-definable continuous deformation retraction
\begin{align*}
 H : I \times \widehat{X} \to \widehat{X}
\end{align*}
whose image $Z$ is 
definably homeomorphic to a definable subset of $\Gamma_\infty^w$ for some finite $A$-definable set
 $w$ such that the following conditions are satisfied.
\begin{enumerate}
 \item Let $R$ denote a finite collection of $A$-definable functions $\xi_i : V \to \Gamma_\infty$.
Then every $\xi_i \in R$ extends to a pro-definable function $\xi_i : \widehat{V} \to V$. We can choose $H$ so that 
$\xi_i \circ H = \xi_i$. This implies that if one were to choose a finite number of subvarieties in $X$ then the homotopy restricts to 
a homotopy of each of these sub varieties. 
\item If $G$ is a finite algebraic group acting on $V$ such that it preserves $X$ then the homotopy can be chosen to be equivariant for this action.
\item Let $I = \left[i_I,e_I\right]$. $H$ satisfies the following property : 
\begin{align*}
  H(e_I,H(t,x)) = H(e_I,x) 
\end{align*}
\item $H$ fixes the image of the homotopy.
\item When $X = V$ and $l = 0$, one may require that the image of the homotopy is 
Zariski dense in $\widehat{V_i}$ for every irreducible component $V_i$ of $V$.
\item Any point in the image, viewed as a stably dominated type has equal transcendence degree and residual transcendence degree. That is
to say that if $p,q$ were two elements of $\widehat{X}$ contained in the image of the homotopy and $M \models ACVF$ which contains $A \cup VF$
and $A \cup \Gamma$ then if 
$c \models p_{|M}$ and $d \models q_{|M}$ then 
\begin{align*}
  \mathrm{trdeg}_M(M(c)) = \mathrm{trdeg}_M(M(d))
\end{align*}
 and similarly
\begin{align*}
  \mathrm{trdeg}_{k(M)} k(M(c)) = \mathrm{trdeg}_{k(M)} k(M(d))
\end{align*}

\end{enumerate}t
 
\end{thm}
     Here $k(M)$ is used to denote the corresponding residue field of the model $M$.      
     
\subsubsection{Restating Theorem 2.11 for Berkovich Spaces}
    
      What follows in this subsection is a restatement of [\cite{loes}, 13.1]. Let $F$ be a valued field such that its value 
group is a subset of $\mathbb{R}_\infty$.  Let $\mathbf{F}$ denote the substructure defined by the pair $(F,\mathbb{R}_\infty)$.   
      Let $V$ be a quasi projective variety and $X$ be an $\mathbf{F}$-definable subset of $V$. We define the 
      Berkovich space $B_{\mathbf{F}}(X)$ to be the set
of $\mathbf{F}$ - types concentrating on the space $X$ which are almost orthogonal [\cite{loes}, 2.4] to $\Gamma$. 
 We define an almost $\Gamma$-orthogonal type as 
follows. 
\begin{defi}
\emph{If $C \subset \mathbb{U}$ and $p$ is a 
$C$-type then we will say that $p$ is} almost orthogonal to $\Gamma$ \emph{if for any realization $a$ of $p$ we have that
 $\Gamma(C(a)) = \Gamma(C)$.}  
\end{defi}
     The space $B_{\mathbf{F}}(X)$ can be endowed with a topology. We proceed as when we defined a topology on $\widehat{X}$.
A pre-basic open subset of $B_{\mathbf{F}}(X)$ is of the form $\{p \in B_{\mathbf{F}}(X \cap O)|\mathrm{val}(f)_*(p) \in W\}$         
where $O \subset V$ is a Zariski open subspace of $V$ defined over $F$, $f$ a regular function on $O$ defined over $F$ and 
$W$ an $\mathbf{F}$ - definable open subset of $\Gamma_\infty$.
 
If $V$ is an $F$-variety 
then $B_{\mathbf{F}}(V)$ is canonically homeomorphic to the Berkovich space $V^{\mathrm{an}}$ discussed previously.
We now explain
how $\widehat{X}$ relates to the space $B_{\mathbf{F}}(X)$.

      Let $K$ be a maximally complete, algebraically closed valued 
 field containing $F$ having value group $\mathbb{R}_{\infty}$ and whose residue field 
 $k(K)$ is the algebraic closure of the residue field $k(F)$ of $F$. Such a field exists and is unique up to 
 isomorphism over the structure $\mathbf{F}$. We will denote such a field $F_{max}$.   
 
 Let $p \in \widehat{X}(F_{max})$. Then $p_{|F_{max}}$ is an $F_{max}$-type,
 and the formulas in $p$ whose parameters lie in $F$ will define an almost 
$\Gamma$-orthogonal $F$-type i.e. an element of the Berkovich space $B_{\mathbf{F}}(X)$. We have thus defined a map 
\begin{align*}
  \pi_X : \widehat{X}(F_{max}) \to B_{\mathbf{F}}(X).
\end{align*}
\begin{prop} [\cite{loes}, 13.1.1]
    If $X$ is an $\mathbf{F}$-definable subset of an algebraic variety then the morphism $\pi_X : \widehat{X}(F_{max}) \to B_{\mathbf{F}}(X)$ is surjective.
\end{prop}
   
     Although Theorem $2.11$ concerns itself with
 the space $\widehat{V}$, it can be deduced that the homotopy constructed induces a homotopy
on the Berkovich space $B_{\mathbf{F}}(V)$ of almost $\Gamma$-orthogonal $\mathbf{F}$-types
 such that the image of this retraction is the space $Z(\mathbf{F})$ [\cite{loes}, Section 13]. This in turn implies the 
 existence of a deformation retraction of $V^{\mathrm{an}}$. 
 Furthermore, $(1)$ (which is the only condition we will use)
 of the above result is also fulfilled
by the induced deformation retraction. 
 
\begin{prop}
   Let $V$ be a quasi-projective variety over a non-Archimedean real valued field $F$.
Let $H : I \times \widehat{V} \to \widehat{V}$ be a pro $\mathbf{F}$ - definable deformation retraction 
whose image $Z := H(e,\widehat{V})$ is definably homeomorphic to a definable subset of $\Gamma_{\infty}^w$ where 
$w$ is a finite $\mathbf{F}$ - definable set.  
Let $\mathbf{I} = I(\mathbb{R}_{\infty})$ and $\mathbf{Z} := \pi_V(Z(F_{max}))$.  
Then $H$ induces a deformation retraction
\begin{align*}
 \tilde{H} : \mathbf{I} \times V^{\mathrm{an}} \to V^{\mathrm{an}}
\end{align*}
whose image $\mathbf{Z}$ is 
homeomorphic to a finite simplicial complex.
\end{prop}

\section{An application of the reduction morphism}  
   
  Our goal in this section is to prove Propositions 3.3 which we will make use of in the proof of Theorem 1.7. We do so using the reduction map described in Section 2.1.1. 
Let $L$ be an algebraically closed,
complete, non-Archimedean 
 real valued field which is non-trivially valued. 
 We will write the value group multiplicatively in this section. 
 Let $A := L[T_1,\ldots,T_n]$ and $I$ be an ideal in this polynomial algebra such that $B :=  L[T_1,\ldots,T_n]/I$ is an integral domain.
Let $U := \mathrm{Spec}(B)$. 
  Let
 $\mathcal{A} := L\{r_1^{-1}T_1,\ldots,r_n^{-1}T_n\}$ and $\mathcal{B} := L\{r_1^{-1}T_1,\ldots,r_n^{-1}T_n\}/I$ be strict
 $L$-affinoid algebras with
 $r_i \in |L^*|$. We use $\mathbf{r}$ to denote the $n$ - tuple $(r_1,\ldots,r_n)$. 
 Let $\pi_{\mathcal{A}}$ and $\pi_{\mathcal{B}}$ denote the reduction morphisms
 $\mathcal{M}(\mathcal{A}) \to \widetilde{\mathcal{M}(\mathcal{A})}$ and
$\mathcal{M}(\mathcal{B}) \to \widetilde{\mathcal{M}(\mathcal{B})}$ respectively. 
We simplify notation and use $X$ to denote the affinoid space $\mathcal{M}(\mathcal{B})$. We assume without loss of generality that 
the space $X$ contains the point at the origin of $\mathrm{Spec}(A)^{\mathrm{an}}$. 
 We have a closed immersion 
 $i : \mathcal{M}(\mathcal{B}) \hookrightarrow \mathcal{M}(\mathcal{A})$. By Section 2.1.1, the reduction morphism 
 is a functorial construction and we have an associated morphism between $\tilde{L}$-schemes
  $\tilde{i} : \widetilde{\mathcal{M}(\mathcal{B})} \to \widetilde{\mathcal{M}(\mathcal{A})}$. Note that this morphism 
  is not necessarily a closed immersion, it is however finite [\cite{BGR}, Theorem 6.3.4/2]. 

  \begin{lem}    
Let $x:= (a_1,\ldots,a_n) \in L^n$ be an $L$-point of the affinoid space $X$. 
Let $\mathbf{B}((a_1,\ldots,a_n), \mathbf{r}) \subset \mathcal{M}(\mathcal{A})$
 denote the Berkovich open ball around $(a_1,\ldots,a_n)$ of poly radius $\mathbf{r}$. 
Let $\tilde{x} :=  \pi_{\mathcal{A}} \circ i(x)$. Let $\tilde{i}^{-1}(\tilde{x}) = \{\tilde{y}_1,\ldots,\tilde{y}_t\}$.  
We have that 
\begin{enumerate} 
\item 
\begin{align*}
\mathbf{B}((a_1,\ldots,a_n), \mathbf{r}) \cap X = \bigcup_i \pi_{\mathcal{B}}^{-1}(\tilde{y}_i).  
\end{align*} 
   The connected components of the open set $\mathbf{B}((a_1,\ldots,a_n), \mathbf{r}) \cap X$ 
   are the $\pi_{\mathcal{B}}^{-1}(\tilde{y}_i)$ for all $i$. 
 \item 
    There exists a finite set of polynomials $ \mathbf{F} := \{F_1,\ldots, F_t\} \subset L[T_1,\ldots,T_n]$  such that 
 for any $y \in X$, the open set 
$\pi_{\mathcal{B}}^{-1}(\pi_{\mathcal{B}}(y)) =  \bigcap_{F \in \mathbf{F}} D_X(F,y)$ where 
$D_X(F,y)$ is the open set $\{p \in X | |(F - F(y))(p)| < 1\}$. 
\end{enumerate}  
   \end{lem} 
  \begin{proof}
\begin{enumerate} 
\item The reduction map $X \mapsto \tilde{X}$ is functorial on the category of 
$L$-affinoid spaces. Hence we have the following commutative 
     diagram. 
     
\begin{center}
       \setlength{\unitlength}{1cm}
\begin{picture}(10,5)
\put(4,1){$\widetilde{\mathcal{M}(\mathcal{B})}$}
\put(7,1){$\widetilde{\mathcal{M}(\mathcal{A})}$}
\put(4,3.5){$\mathcal{M}(\mathcal{B})$}
\put(7,3.5){$\mathcal{M}(\mathcal{A})$}
\put(4.4,3.3){\vector(0,-1){1.75}}
\put(7.4,3.3){\vector(0,-1){1.75}}
\put(5.1,1.1){\vector(1,0){1.82}}
\put(5.1,3.6){\vector(1,0){1.82}}
\put(5.8,0.7){$\tilde{i}$}
\put(5.6,3.2){$i$}
\put(4.5,2.3){$\pi_B$}
\put(7.5,2.3){$\pi_A$}.
\end{picture}        
 \end{center}    
   
    In the diagram, we used $i$ to denote the closed immersion $\mathcal{M}(\mathcal{B}) \hookrightarrow \mathcal{M}(\mathcal{A})$ and 
    $\tilde{i}$ to denote the morphism $ \widetilde{\mathcal{M}(\mathcal{B})} \to \widetilde{\mathcal{M}(\mathcal{A})}$. 
    By the commutativity of the diagram we need only show that 
      $\pi_{\mathcal{A}}^{-1}(\pi_{\mathcal{A}}(x)) =  \mathbf{B}((a_1,\ldots,a_n), \mathbf{r})$.
       
Let $p \in \mathcal{M}(\mathcal{A})$. 
The point $p$ defines a morphism $\mathcal{A} \to \mathcal{H}(p)$ which induces a morphism
$\tilde{\mathcal{A}} \to \widetilde{\mathcal{H}(p)}$. 
Let $h_i \in L$ be such that $|h_i| = r_i^{-1}$. 
It can be verified directly from the definition of $\tilde{\mathcal{A}}$ that $\tilde{\mathcal{A}} = \tilde{L}\left[\tilde{T}_1,..,\tilde{T}_n\right]$ where $\tilde{T}_i$ is the image of 
$h_i.T_i$ for the reduction map $\mathcal{A}^{\circ} \to \tilde{\mathcal{A}}$. The $\tilde{L}$-point $\pi_{\mathcal{A}}(x)$ is defined by 
the $\tilde{L}$-morphism $\tilde{\mathcal{A}} \to \tilde{L}$ which maps $\tilde{T}_i \mapsto \widetilde{h_i.a_i}$.  It follows that 
$\pi_{\mathcal{A}}(p) = \pi_{\mathcal{A}}(x)$ if and only if $|(h_i.T_i - h_i.a_i)(p)| < 1$. Hence we have that 
$\pi_{\mathcal{A}}^{-1}(\pi_{\mathcal{A}}(x)) =  \mathbf{B}((a_1,\ldots,a_n), \mathbf{r})$.   
By [\cite{Bos}, Kor 6.2] and 
Proposition 2.3, the sets $\pi_{\mathcal{B}}^{-1}(\tilde{y}_i)$ are connected and open.

\item  
       By definition, $\mathcal{B}^\circ := \{x \in \mathcal{B} | \rho_\mathcal{B}(x) \leq 1\}$ where 
$\rho_\mathcal{B}$ denotes the spectral norm associated to the affinoid algebra $\mathcal{B}$. 
The ring $\mathcal{B}^\circ$ contains the ideal 
$\mathcal{B}^{\circ \circ} := \{x \in \mathcal{B} | \rho_\mathcal{B}(x) < 1\}$ and we denote the quotient 
$\tilde{\mathcal{B}}$. By [\cite{BGR}, 6.3.4/3], $\tilde{\mathcal{B}}$ is a finite type $\tilde{L}$ - algebra. 
 Let $ \{\tilde{F}_1,\ldots,\tilde{F}_t \}$ be a set of generators of $\tilde{\mathcal{B}}$ and let 
 $\mathbf{F} := \{ F_1, \ldots ,F_t\} $ be a set of elements in $\mathcal{B}^\circ$ such that the image of 
 $F_j$ in $\tilde{\mathcal{B}}$ is $\tilde{F}_j$. We can choose the $F_j$ so that they are polynomials in $L[T_1,\ldots,T_n]$. 
Let $y \in X \cap \mathbf{B}((a_1,\ldots,a_n), \mathbf{r})$ be an $L$ -point. Then 
$\tilde{y} := \pi_{\mathcal{B}}(y) \in \tilde{X}$ is an $\tilde{L}$ - point. The point 
$\tilde{y}$ is uniquely defined by a morphism $\tilde{y} : \tilde{\mathcal{B}} \to \tilde{L}$. This morphism is in 
turn determined by its values at the generators
$\tilde{F}_i$. The image of $\tilde{F}_j$ for the morphism $\tilde{y}$ is 
$\widetilde{F_j(y)} \in \tilde{L}$. Let $p \in X$. 
The point  $\pi_\mathcal{B}(p) \in \tilde{X}$ defines a morphism 
$\tilde{\mathcal{B}} \to \widetilde{\mathcal{H}(p)}$. It follows that 
$\pi_\mathcal{B}(p) = \tilde{y}$ if and only if the images of the $\tilde{F}_j$ for the morphism 
defined by $\pi_\mathcal{B}(p)$ are equal to $\widetilde{F_j(y)}$ i.e. if and only if for every $j$, 
$|(F_j - F_j(y))(p)| < 1$. This proves (2).     
\end{enumerate}
 \end{proof}  
 
       Given a finite set of polynomials $\mathbf{F} \subset L[T_1,\ldots,T_n]$ and a point $x \in X(L)$, we define 
       $\mathfrak{V}_X(\mathbf{F},x) := \cap_{F \in \mathbf{F}} D_X(F,x)$ where $D_X(F,x) = \{p \in X | |(F - F(x))(p)| < 1\}$. 
       We now prove the result which we will use in Section 5.  
 We make use of the notation introduced above. 
 Let $D$ be a finite $B$-algebra
which contains $B$ and is also an integral domain. Hence 
$D = A[S_1,\ldots,S_m]/\langle I, J\rangle$ where 
$J$ is an ideal in the polynomial algebra $C := L[T_1,\ldots,T_n,S_1,\ldots,S_m]$. Let 
$U' := \mathrm{Spec}(D)$. 
Let $\phi$ denote the finite surjective morphism $U' \to U$. It induces a finite 
surjective
morphism $\phi^{\mathrm{an}} :  U'^{\mathrm{an}} \to U^{\mathrm{an}}$.  
The strict affinoid space $X$ is an affinoid domain in $U^{\mathrm{an}}$.
By [\cite{berk3}, 2.1.8, 2.1.9], 
the finiteness of the morphism $\phi^{\mathrm{an}}$ implies that
$\mathcal{B}[S_1,...S_m]/J$ is a strict affinoid algebra and 
\begin{align*}
(\phi^{\mathrm{an}})^{-1}(X) = \mathcal{M}(\mathcal{B}[S_1,...S_m]/J).
\end{align*}
 Let $\mathcal{D}$ denote the affinoid algebra
  $\mathcal{B}[S_1,...S_m]/J$. After a suitable change of coordinates we can assume that 
  $\mathcal{M}(\mathcal{D})$ contains the point at the origin of 
  $(\mathrm{Spec}(C))^{\mathrm{an}}$ and in addition that 
 $\phi$ maps the point at the origin of $\mathcal{M}(\mathcal{D})$ to the origin in 
$X = \mathcal{M}(\mathcal{B})$. We hence assume that   
$\mathcal{D}$ is of the form $L\{r_1^{-1}T_1,\ldots,r_n^{-1}T_n,s_1^{-1}S_1,\ldots,s_m^{-1}S_m\}/\langle I,J \rangle$
where the $s_i$ are non negative real numbers. We define $Y := \mathcal{M}(\mathcal{D})$
and 
 $\mathcal{C} = L\{r_1^{-1}T_1,\ldots,r_n^{-1}T_n,s_1^{-1}S_1,\ldots,s_m^{-1}S_m\}$.
       
       Let $x$ be an $L$-point of 
$X$. 
Let $W$ denote the Berkovich open ball around $x$ of 
poly radius $\mathbf{l} := (l_1,\ldots,l_n)$ contained in $\mathcal{M}(\mathcal{A})$ 
where the $l_i$ are positive real numbers less than or equal to $r_i$ which belong to the value group $|L^*|$.

\begin{rem} Lemma 3.1 implies that there exists a finite set of $L$-points $P  \subset X$
and polynomials $\{F_1,\ldots,F_t\} \subset L[T_1,\ldots, T_n]$ such that 
\begin{align} 
 W \cap X = \bigcup_{x \in P} \mathfrak{V}_{X}(\mathbf{F},x).
 \end{align} 
     Although this is not the exact version of Lemma 3.1, we can derive this formulation as follows. 
   Let $\mathcal{B'} := L\{l_1^{-1}T_1,\ldots,l_n^{-1}T_n\}/I$. 
     By Lemma 3.1, we have that 
 \begin{align*}
 W \cap X = \bigcup_{x \in P} \mathfrak{V}_{\mathcal{M}(\mathcal{B}')}(\mathbf{F}',x).
 \end{align*}     
     for some finite set of polynomials $\mathbf{F}'$ and a finite set of $L$-points $P$. For every $l_i$ let $e_i \in L$ be such that $|e_i| = l_i^{-1}$. 
 Extending the set $\mathbf{F}'$ by adding the polynomials $e_iT_i$ will yield equation (2).     
The set $P$ is chosen so that the right hand side of (2) is the disjoint union of open sets.   
\end{rem} 

Proposition $3.3$ below concerns itself with the nature of the 
preimage $(\phi^{\mathrm{an}})^{-1}(W \cap X)$. 
           
 \begin{prop}
    There exists a finite set of 
    polynomials $\mathbf{G} := \{G_1,\ldots,G_{t'}\} \subset L[T_1,\ldots,T_n,S_1,\ldots,S_m]$
    and a finite set of points $Q \subset Y(L)$ 
     such that        
  \begin{enumerate}
 \item    We have the following equality of sets 
\begin{align*}
 (\phi^{\mathrm{an}})^{-1}(W \cap X) = 
  \bigcup_{y \in  Q} \mathfrak{V}_Y(\mathbf{G},y).
\end{align*}  
 \item  The $\{\mathfrak{V}_Y(\mathbf{G},y)\}_{y \in Q}$ are the connected components of the
         space $(\phi^{\mathrm{an}})^{-1}(W \cap X)$. 
 \item When the restriction of the morphism $\phi^{\mathrm{an}}$ to the open set  $(\phi^{\mathrm{an}})^{-1}(W \cap X)$ is an open morphism,
         we can choose $Q$ to be the set $\phi^{-1}(P)$ and for any $y \in \phi^{-1}(P)$, we have that 
         $\phi^{\mathrm{an}}(\mathfrak{V}_Y(\mathbf{G},y)) = \mathfrak{V}_X(\mathbf{F},\phi(y))$. 
  \end{enumerate}
  \end{prop}

\begin{proof}
  The Berkovich open ball $W \subset \mathcal{M}(\mathcal{A})$ has poly radius 
$(l_1,\ldots,l_n)$ with  
$l_i \in |L^*|$ for every $i$. 
We will assume without loss of generality that the point $x$ has coordinates $(0,\ldots,0)$.  
Let $\mathcal{B}' := L\{l_1^{-1}T_1,\ldots,l_n^{-1}T_n\}/I$.
Observe that the affinoid space $\mathcal{M}(\mathcal{B}')$ is the intersection of the Berkovich closed disk 
centred at $x$ in $\mathcal{M}(\mathcal{A})$ of poly radius $(l_1,..,l_n)$ and 
$(\mathrm{Spec}(B))^{\mathrm{an}}$. 
Let $\mathcal{D}' := \mathcal{B}'[S_1,..,S_m]/J$. 
By [\cite{berk3}, 2.1.8, 2.1.9], 
$\mathcal{D}'$ is a strict $L$-affinoid algebra since it is a finite $\mathcal{B}'$-algebra.
By definition $\mathcal{D}'$ contains $\mathcal{B}'$.  
Furthermore,
\begin{align*}
   (\phi^{\mathrm{an}})^{-1}(\mathcal{M}(\mathcal{B}')) = \mathcal{M}(\mathcal{D}').
\end{align*}
    As with the affinoid algebra $\mathcal{D}$, we can write
\begin{align*}
\mathcal{D}' = L\{l_1^{-1}T_1,\ldots,l_n^{-1}T_n,{l'}_1^{-1}S_1,\ldots,{l'}_m^{-1}S_m\}/\langle I,J\rangle
\end{align*}
   where the $l'_i$ are non negative real numbers belonging to $|L^*|$. 

Consider the following commutative diagram.  
    
   \setlength{\unitlength}{1cm}
\begin{picture}(10,5)
\put(4,1){$\widetilde{\mathcal{M}(\mathcal{D}')}$}
\put(7,1){$\widetilde{\mathcal{M}(\mathcal{B}')}$}
\put(4,3.5){$\mathcal{M}(\mathcal{D}')$}
\put(7,3.5){$\mathcal{M}(\mathcal{B}')$}
\put(4.4,3.3){\vector(0,-1){1.75}}
\put(7.4,3.3){\vector(0,-1){1.75}}
\put(5.1,1.1){\vector(1,0){1.82}}
\put(5.1,3.6){\vector(1,0){1.82}}
\put(5.8,0.7){$\widetilde{\phi^{\mathrm{an}}}$}
\put(5.6,3.2){$\phi^{\mathrm{an}}$}
\put(4.5,2.3){$\pi_{\mathcal{D}'}$}
\put(7.5,2.3){$\pi_{\mathcal{B}'}$}
\end{picture}   

    The morphism $\mathcal{M}(\mathcal{D}') \to \mathcal{M}(\mathcal{B}')$ is finite. By [\cite{BGR}, Theorem 6.3.4/2], 
    the induced morphism between the associated reductions 
    $\widetilde{\mathcal{M}(\mathcal{D}')} \to \widetilde{\mathcal{M}(\mathcal{B}')}$ is finite as well.
    For every $x_i \in P$, let $\tilde{x}_i$ be the image of $x_i$ for the reduction morphism 
   $\pi_{\mathcal{B}'}$. 
   Let $\tilde{Q} := \{\tilde{z}_1,\ldots,\tilde{z}_v\} \subset \widetilde{\mathcal{M}(\mathcal{D}')}$ be the set of preimages of the set 
   $\{\tilde{x}_i | x_i \in P\}$ for the morphism $\widetilde{\phi^{\mathrm{an}}}$.          
   From the commutative diagram above, we have the following equality 
\begin{align}
   \bigcup_{\tilde{z}_i \in \tilde{Q}}  \pi_{\mathcal{D}'}^{-1}(\tilde{z}_i)  =  
    \bigcup_{x_j \in P}  (\phi^{\mathrm{an}})^{-1}(\pi_{\mathcal{B}'}^{-1}(\pi_{\mathcal{B}'}(x_j))).
\end{align}
  By Proposition 2.3 and [\cite{Bos}, Kor 6.2], the sets $\pi_{\mathcal{D}'}^{-1}(\tilde{z}_i) \subset \mathcal{M}(\mathcal{D}')$
  are connected and open. From the commutative diagram we can also infer the following inequality
  \begin{align}
   \bigcup_{y_i \in \phi^{-1}(P)} \pi_{\mathcal{D}'}^{-1}(\pi_{\mathcal{D}'}(y_i)) \subseteq \bigcup_{\tilde{z}_i \in H}  \pi_{\mathcal{D}'}^{-1}(\tilde{z}_i).
\end{align}
  
  Let $Q$ be a set of $L$ - points of $\mathcal{M}(\mathcal{D'})$ which are in bijection with the set $\tilde{Q}$ via the reduction morphism 
  $\pi_{\mathcal{D}'}$. 
  By Remark 3.2, Lemma 3.1 and equation (3) there exists a finite set of polynomials $\mathbf{F} \subset L[T_1,\ldots,T_n]$ and 
$\mathbf{G}_0 \subset L[T_1,\ldots,T_n,S_1,\ldots,S_m]$ such that 
\begin{align*} 
   W \cap X = \bigcup_{x \in P} \mathfrak{V}_{\mathcal{M}(\mathcal{B}')}(\mathbf{F},x) = \bigcup_{x \in P} \mathfrak{V}_{X}(\mathbf{F},x)
\end{align*}   
and 
\begin{align*} 
  (\phi^{\mathrm{an}})^{-1} (W \cap X) = \bigcup_{z \in Q} \mathfrak{V}_{\mathcal{M}(\mathcal{D}')}(\mathbf{G}_0,z).
\end{align*}
   To complete the proof of the proposition, we enlarge the set $G_0$ as follows. Recall that in Remark 3.2, we chose 
   $e_i \in L$ such that $|e_i| = l_i^{-1}$. Likewise let $|e'_i| = {l'}_i^{-1}$. Such elements exists as we had 
   assumed that $l_i$ and $l'_i$ belong to $|L^*|$. Let $T'_i := e_iT_i$ and $S'_i := e'_iS_i$. 
   Observe that since
   $\mathfrak{V}_{\mathcal{M}(\mathcal{D}')}(\mathbf{G}_0,y) = \pi_{\mathcal{D}'}^{-1}(\pi_{\mathcal{D}'}(y))$, 
   by Lemma 3.1(1)  
   we must have that $\mathfrak{V}_{\mathcal{M}(\mathcal{D}')}(\mathbf{G}_0,y) \subseteq B(y,(l_1,\ldots,l_n,l'_1,\ldots,l'_m))$
   where $B(y,(l_1,\ldots,l_n,l'_1,\ldots,l'_m))$ is the Berkovich open ball in $\mathcal{M}(\mathcal{C})$ around $y$ of poly radius 
   $(l_1,\ldots,l_n,l'_1,\ldots,l'_m)$. It follows that 
   \begin{align*}
   \mathfrak{V}_{\mathcal{M}(\mathcal{D}')}(\mathbf{G}_0,y) = \mathfrak{V}_{Y}(\mathbf{G},y).
   \end{align*} 
       This concludes parts (1) and (2) of the proposition. 
   
    We now show that the inequality (4) above which is 
    \begin{align*}
   \bigcup_{y_i \in \phi^{-1}(P)} \pi_{\mathcal{D}'}^{-1}(\pi_{\mathcal{D}'}(y_i)) \subseteq \bigcup_{\tilde{z}_i \in H}  \pi_{\mathcal{D}'}^{-1}(\tilde{z}_i)
\end{align*}
    is in fact an equality when the restriction $\phi^{\mathrm{an}}$ to 
   $(\phi^{\mathrm{an}})^{-1}(W \cap X)$ is an open morphism. 
   In this case, we claim that if 
   $\tilde{z}_i \in \tilde{Q}$ then there exists $x_j \in P$ such that $\phi^{\mathrm{an}}$ restricts to a 
   surjection from
   $\pi_{\mathcal{D}'}^{-1}(\tilde{z}_i)$ onto 
   $\pi_{\mathcal{B}'}^{-1}(\pi_{\mathcal{B}'}(x_j))$. The morphism 
   $\phi^{\mathrm{an}}$ restricts to a morphism between $\pi_{\mathcal{D}'}^{-1}(\tilde{z}_i)$
   and the disjoint union of open sets - $\bigcup_{x \in P} \pi_{\mathcal{B}'}^{-1}(\pi_{\mathcal{B}'}(x))$. The connectedness of $\pi_{\mathcal{D}'}^{-1}(\tilde{z}_i)$ implies 
   that there exists $x_j \in P$ such that the image $\phi^{\mathrm{an}}((\pi_{\mathcal{D}'})^{-1}(\tilde{z}_i))$ is contained in 
   $\pi_{\mathcal{B}'}^{-1}(\pi_{\mathcal{B}'}(x_j))$. Also,
   the restriction $\phi^{\mathrm{an}} : \mathcal{M}(\mathcal{D'}) \to \mathcal{M}(\mathcal{B'})$ is closed as it is a finite morphism.
    The set $\pi_{\mathcal{D}'}^{-1}(\tilde{z}_i)$ is both open and closed in 
   $\bigcup_{\tilde{z}_i \in \tilde{Q}}  \pi_{\mathcal{D}'}^{-1}(\tilde{z}_i)$.
   As $\pi_{\mathcal{B}'}^{-1}(\pi_{\mathcal{B}'}(x_j))$ is connected, we must have that 
    $\phi^{\mathrm{an}}$ restricts to a 
   surjection from
   $\pi_{\mathcal{D}'}^{-1}(\tilde{z}_i)$ onto 
   $\pi_{\mathcal{B}'}^{-1}(\pi_{\mathcal{B}'}(x_j))$. 
    It follows that there exists $y \in \phi^{-1}(P)$ such that $y \in \pi_{\mathcal{D}'}^{-1}(\tilde{z}_i)$ from which we get the equality   
    \begin{align*}
   \bigcup_{y_i \in \phi^{-1}(P)} \pi_{\mathcal{D}'}^{-1}(\pi_{\mathcal{D}'}(y_i)) = \bigcup_{\tilde{z}_i \in \tilde{Q}}  \pi_{\mathcal{D}'}^{-1}(\tilde{z}_i).
\end{align*}
   This proves part (3) of the proposition. 
\end{proof}

\section{The theorem for $\widehat{V}$} 

       We now reinterpret Theorem 1.7 for the spaces $\widehat{V}$ discussed in Section 2.3.      
  Let $V'$ and $V$ be integral projective $k$ - varieties such that $V$ is normal and let 
  $\phi : V' \to V$ be a finite surjective morphism. The morphism $\phi$ induces a pro - definable map 
  $\widehat{\phi} : \widehat{V}' \to \widehat{V}$. We write the structure of the value group additively in this section. 
  
         As before, we fix an embedding 
  $V \hookrightarrow \mathbb{P}^n_k$ and an affine chart of $\mathbb{P}^n_k$.    
  We will regard the spaces $V'$, $V$ and $\mathbb{P}^n$ as $k$ - definable sets in ACVF. 
  As in Section 2, we fix $\mathbb{U}$ -  a very large saturated model of ACVF and assume that every model of interest to us 
 is a small sub structure of $\mathbb{U}$.  
 
 One of the advantages of working in the model theoretic setting is that we need no longer concern ourself
 with the process of extending scalars.
 In Section 2.1 of \cite{loes} a brief discussion concerning definable sets is given. We reproduce a part of that discussion here.  
  Let $\sigma$ be a formula in ACVF with parameters contained in a structure $C$. 
 The formula $\sigma$ defines a functor from the category of models and elementary embeddings of ACVF which contain $C$ to the 
 category of sets i.e. given a model $M$ of ACVF which contains $C$, the functor $Z_{\sigma}$ associates $M$ to the 
 set $Z_{\sigma}(M) := \{a \in M | M \models \sigma(a)\}$. The functor  $Z_{\sigma}$ is completely determined by the large set 
 $Z_{\sigma}(\mathbb{U})$. The set of $L$ - points of $V \times_k L$ in the algebraic sense is the set $V(L)$ 
 (in the model theoretic sense)
 where the latter is 
 not to be confused with the scheme theoretic notion $\mathrm{Hom}_k(\mathrm{Spec}(L),V)$.
 The points which are not closed in the variety $V$ correspond to $k$ - types which concentrate on $V$.       
 
 As in 2.6, the definable set $\mathbb{P}^n$ can be realised as the union of $n + 1$ 
 closed disks $A_i^0$ which are glued together definably. Each 
 $A_i^0$ is a $0$ - definable sub set of $\mathbb{A}^n$ and comes equipped with definable functions 
 $T_j/T_i : A_i^0 \to VF$ for $j \in \{1,\ldots,n+1\}$ and where $VF$ denotes the value field sort. 
These functions define the coordinates of the points of $A_i^0$. In 2.2, we used these functions to define 
 Berkovich open balls around points. We repeat that procedure to define for every 
 $x \in \mathbb{P}^n$, a family of definable sets $\mathcal{O}_x^0$ which are $v + g$ open neighbourhoods of 
 $x$.  The notion of a $v + g$ topology was introduced in [\cite{loes}, Section 3.7]. Explicitly,
 if $O \in  \mathcal{O}_x^0$ and $x \in A_i^0$ for some $i$ then 
 $O$ is defined by the formula $\{x \in \mathbb{P}^n | val(T_j/T_i - T_j/T_i(x)) > r_j \}$ where 
 $r_j \in \Gamma$, $r_j \geq 0$ and $val$ denotes the valuation $VF \to \Gamma_{\infty}$. It can be checked that the family $\mathcal{O}_x^0$ defined in 
 this way is independent of the $A^0_i$ chosen. In addition, we have as before a function $h : \mathcal{O}_x^0 \to [0,\infty]^{{n+1}^2}$ which defines the 
 poly radii of elements of $\mathcal{O}_x^0$. Precisely, if $x \in A_i^0$ and 
 $O \in \mathcal{O}_x^0$ then $O$ is uniquely defined by its poly radius $\mathbf{r}_i := (r_1,\ldots,r_{n+1})$ where 
 we set $r_i = 0$. If $x \notin A_i^0$ then we set $\mathbf{r}_i := (0,\ldots,0)$. 
 We define $h(O) := (\mathbf{r}_i)_i$. Observe that if $x$ and $h(O)$ are defined over a model $M$ of ACVF then 
 $\widehat{O}$ is an open pro - $M$ definable subspace of $\widehat{\mathbb{P}}^n$.
By definition the function $h$ extends to a map $h : \mathcal{O}^0 := \bigcup_{x \in \mathbb{P}^n} \mathcal{O}^0_x \to  [0,\infty]^{{n+1}^2}$. 
Observe that for $x \in \mathbb{P}^n(L)$ where $L$ is a non-Archimedean real valued extension of $k$ and 
$O \in \mathcal{O}^0_x(\mathbb{R}_{\infty})$, we have that $h(O) = \mathrm{log}(h_L(O^{\mathrm{an}}))$ (Remark 1.3, Section 2.2). 
  Let $R \subset \Gamma_{\infty}^{{n+1}^2} \times V$ be defined 
 by those pairs $(\mathbf{r},x)$ for which there exists an element $O \in \mathcal{O}_x^0$ such that 
 $h(O) = \mathbf{r}$. The set $R$ is $k$ - definable. This can be shown by an explicit calculation similar to what was done in Section 2.2.1for 
 $\mathbb{P}_k^{1}$.

\begin{rem}   In Remark 1.5, we introduced a collection $S$ of functions from $\mathbb{R}_{>0}^{(n+1)^2}$ to $\mathbb{R}_{ > 0}$ an element 
   of which extends to a function $\mathbb{R}_{\geq 0}^{(n+1)^2} \to \mathbb{R}_{ \geq 0}$ naturally. When writing the value group 
   additively, we adapt the family $S$ as follows. Firstly, $\mathrm{log} : (\mathbb{R}_{> 0}, \times) \to (\mathbb{R}, +)$ (Remark 1.3) 
   is an isomorphism of abelian groups which reverses the ordering and whose inverse is 
   the function $\mathrm{exp} : (\mathbb{R},+) \to (\mathbb{R}_{> 0}, \times)$
    which maps $x \mapsto c^x$. If $g \in S$, we define $g' : \mathbb{R}^{{n+1}^2} \to \mathbb{R}$. 
   Let $\mathbf{r} = (r_{i,j})_{i,j} \in \mathbb{R}^{{n+1}^2}$. It follows that $c^{\mathbf{r}} \in \mathbb{R}_{>0}^{(n+1)^2}$ where 
   $c^{\mathbf{r}} := (c^{r_{i,j}})_{i,j}$. Let $g'(\mathbf{r}) := \mathrm{log}(g(c^{\mathbf{r}}))$. The properties of the function $g \in S$ imply the following. 
   
\begin{enumerate}
\item The function $g'$
 is continuous with respect to the topology induced by the ordering. 
 \item If $(r_{i,j})_{i,j}$ and $(s_{i,j})_{i,j}$ are $(n+1)^2$-tuples in $\mathbb{R}^{(n+1)^2}$  such that 
$r_{i,j} \leq s_{i,j}$ then $g'((r_{i,j})_{i,j}) \leq g'((s_{i,j})_{i,j})$.
\item $g$ is a definable function in the language of Ordered Abelian groups.
 \end{enumerate}
     As in Remark 1.5, we extend the function $g'$ so that it defines a function 
    $\Gamma_{\infty}^{{n+1}^2} \to \Gamma_{\infty}$. 
 
 \end{rem}
 
     Let
  $g \in S$. As in Section 2.3, given a real valued model $F$ of ACVF, let $\mathbf{F}$ denote the structure defined by the pair 
  $(F,\mathbb{R}_{\infty})$. 
  As in Section 2.2, the function 
 $g'$ induces an ordering on the set $\mathcal{O}^0(\mathbf{F}) := \bigcup_{x \in \mathbb{P}^n(F)} \mathcal{O}^0_x(\mathbf{F})$ where 
  $\mathcal{O}^0_x(\mathbf{F})$ are those elements of  $\mathcal{O}^0_x$ which are defined over $\mathbf{F}$. 
  More precisely, as in Lemma 2.8, 
  the function 
$g' \circ h : \mathcal{O}^0(\mathbf{F}) \to \mathbb{R}_{\infty}$ has the following property.  
 If 
$O_1, O_2 \in \mathcal{O}^0(\mathbf{F})$ such that $O_1 \subseteq O_2$ then     
$(g' \circ h)(O_1) \geq (g' \circ h)(O_2)$. 
The inequality above has been reversed owing to the fact that $h(O_i) = \mathrm{log}(h_F(O^{\mathrm{an}}_i))$. 
The functions $g \in S$ hence allow us to quantify the size of elements belonging to $\mathcal{O}^0$. 

\begin{lem}
  Let $d$ denote the separable degree of the finite morphism $\phi : V' \to V$. 
  For $p \in \widehat{V}$, the cardinality of the set of preimages $\widehat{\phi}^{-1}(p)$ is bounded above by $d$ 
  and the set of simple points $x$ in $V$ for which $\mathrm{card}(\phi^{-1}(x)) = d$ is dense in $\widehat{V}$.   
\end{lem}
   \begin{proof} 
      Let $M$ be a model of ACVF which contains $k$ and $x \in V(M)$. The point $x$ defines a closed point of the variety $V \times_k M$. 
     From algebraic geometry, the cardinality of fibre $\phi_M^{-1}(x)$ is bounded above by $d$ where $\phi_M : V' \times_k M \to V \times_k M$.  
   From our discussion above on definable sets, we conclude that the cardinality of $\phi^{-1}(x)$ is also bounded
    above by $d$ when $\phi$ is viewed as a definable map between the definable sets $V'$ and $V$.   
  
   Let $p \in \widehat{V}$. By definition, $p$ is a stably dominated type which concentrates on $V$. 
   Let us assume that it is defined over a model $M$ of ACVF which contains $k$. Let $a$ be a realization of the 
   $M$ - type $p_{|M}$. Our discussion above implies that there exists $\{a'_1,\ldots,a'_t\} \subset V'$ such that 
   $\phi^{-1}(a) :=  \{a'_1,\ldots,a'_t\}$ and $t \leq d$. As the function $\phi$ is definable over $k$, we must have that the preimage of the 
   $M$ - type $p_{|M}$ extends to at most $t$ $M$ - types which concentrate on $V'$ and hence at most $t$ stably dominated types over $M$. 
  
    We now verify the remainder of the lemma. 
    By Lemma 4.8, there exists affine open sets $U \subset V$, $U' \subset V'$ and an affine scheme $U''$ along 
  with morphisms $\phi_1 : U' \to U''$ purely inseparable and $\phi_2 : U'' \to U$ separable of degree $d$ such that
  the restriction of the morphism $\phi$ to $U'$ factors as $\phi_2 \circ \phi_1$.  There 
  exists an open sub scheme $U_0 \subset U$ over which the morphism $\phi_2$ is etale and 
  since $k$ is algebraically closed, the cardinality of the set $\phi_2^{-1}(y)$ for $y \in U_0$
  is equal to $d$. By [\cite{hart}, 3.15], the scheme $U \times_k L$ is irreducible for any field extension $L$ of $k$. 
  Hence $U_0 \times_k L$ is Zariski dense in $U \times_k L$. 
  
  We claim that $U_0$ is dense in $\widehat{V}$. 
  Indeed, let $f$ be a non constant regular function of some affine open subspace $W$ of $V$ and 
  $E$ some open set in $\Gamma_{\infty}$. From our discussion above, the definable set $U_0 \cap W$ is not empty and 
  $f$ is regular when restricted to $U_0 \cap W$ and non constant.   
  The function $f : U_0 \cap W \to \mathbb{A}^1$ is dominant for the Zariski topology. Let $\bar{f} : \bar{U}_0 \to \mathbb{P}^1$ be a compactification of 
  the morphism $f$. There exists an open set $W_0 \subset \mathbb{A}^1$ over which $\bar{f}$ is flat and hence open. It follows that the 
  image of the map $f$ contains a Zariski open subset of $\mathbb{A}^1$. 
     We hence reduce to showing that 
  if $W \subset \mathbb{A}^1$ is a Zariski dense open subset, then the set $\{val(x) | x \in W\} \subset \Gamma_{\infty}$ is 
  dense for the linear topology. This can be easily verified. 
    Hence we have that the open set $\widehat{O} \cap \widehat{V}$ contains at least 
 one point which belongs to $U_0$. This proves the claim and concludes the proof.

   \end{proof} 
 
       For $O \in \mathcal{O}^0_x$ with $h(O) = \mathbf{r}$, let $N_{V}(\mathbf{r},x)$ denote the 
       number of connected components of the space $\widehat{O} \cap \widehat{V}$ and 
       $N_{V'}(\mathbf{r},x)$ denote the number of connected components of the 
       space $\widehat{\phi}^{-1}(\widehat{V} \cap \widehat{O})$.

\begin{lem}
   Let $d$ denote the separable degree of the morphism $\phi : V' \to V$. Let $x \in V$ and 
   $O \in \mathcal{O}^0_x$. Let $\mathbf{r} := h(O)$. The preimage $\widehat{\phi}^{-1}(\widehat{O} \cap \widehat{V})$ is the disjoint union of 
   $\widehat{V}'$ open sets each of which is homeomorphic to $\widehat{O} \cap \widehat{V}$ via $\widehat{\phi}$ if and only if 
   $N_{V'}(\mathbf{r},x) = d.N_{V}(\mathbf{r},x)$. Furthermore, if $\widehat{\phi}^{-1}(\widehat{O} \cap \widehat{V})$ is the disjoint union of 
   $\widehat{V}'$ open sets each of which is homeomorphic to $\widehat{O} \cap \widehat{V}$ via $\widehat{\phi}$ then 
   the cardinality of the number of preimages of a point in $\widehat{O} \cap \widehat{V}$ is $d$. 
\end{lem}
\begin{proof}
As the variety $V$ is normal, by 
 Corollary 8.7.2 in \cite{loes}, the morphism $\widehat{\phi} : \widehat{V}' \to \widehat{V}$ is open. 
 By Lemma 4.2.24 in loc.cit, the map $\widehat{\phi}$ is closed as well. 
Let $(\mathbf{r},x) \in R$ and $O \in \mathcal{O}_x^0$ such that $h(O) = \mathbf{r}$.
By 10.1.1, there exists a continuous deformation retraction $H : I \times \widehat{O} \cap \widehat{V} \to \widehat{O} \cap \widehat{V}$
such that the image $H(e,\widehat{O} \cap \widehat{V})$ is a $\Gamma$ - internal subset of $\widehat{O} \cap \widehat{V}$ which 
is definably homeomorphic to a definable subset $\Lambda$ in $\Gamma_{\infty}^m$ for some $m \in \mathbb{N}$. 
The connected components of $\Lambda$ are open and there are only finitely many of these. It follows that 
$\widehat{O} \cap \widehat{V}$ is the finite disjoint union of path connected open sets.  
Let $\{C_1,\ldots,C_t\}$ be the connected components of  $\widehat{O} \cap \widehat{V}$. By 
a similar argument, $\widehat{\phi}^{-1}(\widehat{O} \cap \widehat{V})$ is the disjoint union of a finite number of open 
sets each of which are path connected. Let $\{C'_1,\ldots,C'_{t'}\}$ denote the connected components of $\widehat{\phi}^{-1}(\widehat{O} \cap \widehat{V})$. 
As the morphism $\widehat{\phi}$ is clopen when restricted $\widehat{\phi}^{-1}(\widehat{O} \cap \widehat{V})$, we see that for every $j$ there exists a unique $i$ such that  
$\widehat{\phi}$ maps $C'_j$ surjectively onto $C_i$. Hence 
the preimage $\widehat{\phi}^{-1}(\widehat{O} \cap \widehat{V})$ is the disjoint union of 
   $\widehat{V}'$ open sets each of which is homeomorphic to $\widehat{O} \cap \widehat{V}$ via $\widehat{\phi}$ if and only if 
  for every $j$ there exists a unique $i$ such that $\widehat{\phi}$ restricts to a bijection 
  from $C'_j$ onto $C_i$ and the number of preimages of an element $p \in \widehat{O} \cap \widehat{V}$ is constant. 
  By Lemma 4.2, this constant must be $d$ as the set of elements $p \in \widehat{V}$ for which 
  $\mathrm{card}(\widehat{\phi}^{-1}(p)) = d$ is dense in $\widehat{V}$.
  Since by Lemma 4.2, for $p \in \widehat{V}$ the set $\widehat{\phi}^{-1}(p)$ has cardinality bounded by $d$  
   it follows that 
  the preimage $\widehat{\phi}^{-1}(\widehat{O} \cap \widehat{V})$ is the disjoint union of 
   $\widehat{V}'$ open sets each of which is homeomorphic to $\widehat{O} \cap \widehat{V}$ via $\widehat{\phi}$ if and only if $t' = dt$ i.e 
   $N_{V'}(\mathbf{r},x) = dN_{V}(\mathbf{r},x)$. 
  
\end{proof} 

     Te proof above can be adapted to show the following result. 
\begin{lem}
Let $M$ be a model of ACVF.  
       Let $(\mathbf{r},x) \in R(M)$ and $O \in \mathcal{O}^0_x$ such that $h(O) = \mathbf{r}$. 
        Let $N^M_{V'}(\mathbf{r},x)$ denote the number of connected components of the
       space $(\widehat{\phi}^{-1})(\widehat{O}(M) \cap \widehat{V}(M))$ and 
       $N^M_V(\mathbf{r},x)$ be the number of connected components of the space 
       $\widehat{O}(M) \cap \widehat{V}(M)$. The preimage $\widehat{\phi}^{-1}(\widehat{O}(M) \cap \widehat{V}(M))$ is the disjoint union of 
   $\widehat{V}'(M)$ open sets each of which is homeomorphic to $\widehat{O}(M) \cap \widehat{V}(M)$ via $\widehat{\phi}(M)$ if and only if 
   $N^M_{V'}(\mathbf{r},x) = d.N^M_{V}(\mathbf{r},x)$. 
\end{lem}

       We now prove a lemma which is central to the proof of Theorem 4.6. 
     We preserve the notation $N_{V'}^M(\mathbf{r},x)$ and 
      $N^M_V(\mathbf{r},x)$ introduced in the preceding lemma.  
      
       \begin{lem} 
          There exists a $k$ - definable subset $D$ of $R$ such that 
        for $M$ a model of ACVF with value group $\mathbb{R}_{\infty}$
    we have that $(\mathbf{r},x) \in D(M)$ if and only if $N^M_{V'}(\mathbf{r},x) = dN^M_V(\mathbf{r},x)$ where $d$ denotes the separable degree of the 
    morphism $\phi$.      
         \end{lem}          
      \begin{proof} 
        Consider the set $X \subset V' \times R$ consisting of tuples $(z,\mathbf{r},x)$
        such that 
         $(\mathbf{r},x) \in R$, $z \in V'$ and $\phi(z) \in O$ where 
         $O \in \mathcal{O}_{x}^0$ and $h(O) = \mathbf{r}$. The set $X$ is definable 
         and if $pr$ denotes the projection $X \to R$ then for any $(\mathbf{r},x) \in R$    
     the fibre $pr^{-1}(\mathbf{r},x)$ will be the definable set $\phi^{-1}(O \cap V)$ where $O \in \mathcal{O}^0_x$ such that 
     $h(O) = \mathbf{r}$. For $\tau = (\mathbf{r},x) \in R$, we will write $O_\tau$ for that element $O \in \mathcal{O}_{x}^0$ such that 
     $h(O) = \mathbf{r}$. 
     Given $\tau \in R$, we write $k(\tau)$ for the definable closure of $k \cup \{\tau\}$. 
      By Theorem 10.7.1 in \cite{loes} 
       there exists, uniformly in $\tau \in R$,  
      a pro - definable family $H_{\tau}  : I \times \widehat{\phi}^{-1}(\widehat{O}_{\tau} \cap \widehat{V}) \to \widehat{\phi}^{-1}(\widehat{O}_{\tau} \cap \widehat{V})$, a finite 
   $k(\tau)$ -definable set $w(\tau)$, a $k(\tau)$-  definable set $W_{\tau} \subset \Gamma_{\infty}^{w(\tau)}$ 
         and $j_{\tau} : W_{\tau} \to H_{\tau}(e,  \widehat{\phi}^{-1}(\widehat{O}_{\tau} \cap \widehat{V}))$ , pro - definable uniformly in $\tau$ such
that for each $\tau \in R$, $H_{\tau}$ is a deformation retraction and $j_{\tau} : W_{\tau} \to H_{\tau}(e,\widehat{\phi}^{-1}(\widehat{O}_{\tau} \cap \widehat{V}))$ is
a definable homeomorphism and $e$ denotes the end point of the interval $I$.  Let $Z_{\tau} :=  H_{\tau}(e,  \widehat{\phi}^{-1}(\widehat{O}_{\tau} \cap \widehat{V})$.  
     By the claim in the proof of Theorem 13.3.1 \cite{loes}, there exists uniformly in $\tau$ a $k(\tau)$ - definable 
     set $T_{\tau} \subset \Gamma_{\infty}^r$, a $k(\tau)$ - definable set $W(\tau)$ and for 
     $w \in W(\tau)$, a definable homeomorphism $\psi_{w} : Z_{\tau} \to T_{\tau}$.   
     We have in this manner obtained a family of definable subsets of $\Gamma_{\infty}^r$ parametrized by $R$. 
     Observe that if $M$ is a model of ACVF with value group $\mathbb{R}_{\infty}$ and 
     $\tau \in R(M)$ then the image of the deformation retraction 
     $H_{\tau}(M) : I(\mathbb{R}_{\infty}) \times \widehat{\phi}^{-1}(\widehat{O}_{\tau}(M) \cap \widehat{V}(M)) \to \widehat{\phi}^{-1}(\widehat{O}_{\tau}(M) \cap \widehat{V}(M))$ is 
     definably homeomorphic to 
     $T_{\tau}(\mathbb{R}_{\infty})$. 
   
       Let $\Gamma^*$ be an expansion of $\Gamma$ to RCF and ACVF' denote the extension of ACVF with the sort $\Gamma^*$ in place of $\Gamma$. 
           By remark 13.3.2 in \cite{loes}, there exists a  
     finite number of polytopes such that for every $\tau \in R$, $T_\tau(\mathbb{R}_{\infty})$ is  
     homeomorphic to exactly one of these polytopes. This implies firstly that there exists $N$ such that for every $\tau \in R$,  
     $T_{\tau}(\mathbb{R}_{\infty})$ has less than $N$ connected components. 
     It follows that $R$ can be partitioned into $N$ ACVF' definable sets $\{E'_1,\ldots,E'_N\}$  
      such that 
     if $M$ is a model of ACVF whose value group is $\mathbb{R}_{\infty}$ and 
     $\tau := (\mathbf{r},x) \in E'_j(M)$ then $\widehat{\phi}^{-1}(\widehat{O}(M) \cap \widehat{V}(M))$ must have 
     $j$ connected components.  The $E'_j$ are ACVF' definable with parameters in $k$. Indeed, the field 
     $k_{max}$ which is a maximally complete field extension of $k$ with value group $\mathbb{R}_{\infty}$ and
     residue field $\tilde{k}$ (Section 2.3) can be extended to a model of ACVF'. Let $g \in \mathrm{Aut}(k_{max}/k)$. 
     Let $\tau \in R(k_{max})$. As $R$ is definable over $k$ and the $T_{\tau}$ are definable uniformly over 
     $k(\tau)$, it follows that $T_{\tau}(\mathbb{R}_{\infty})$ and 
    $T_{g(\tau)}(\mathbb{R}_{\infty})$ have the same number of connected components which implies $g({\tau}) \in E'_j(k_{max})$. 
     As $E'_j(k_{max})$ is preserved by the action of $\mathrm{Aut}(k_{max}/k)$, we conclude that it is defined with parameters from $k$.  
     
    In the above discussion if we were to substitute the set $X \subset V' \times R$ with the $k$ - definable set 
    $X' \subset V \times R$ defined by 
    tuples $(z,\mathbf{r},x)$
        such that 
         $(\mathbf{r},x) \in R$, $z \in V \cap O$ where 
         $O \in \mathcal{O}_{x}^0$ and $h(O) = \mathbf{r}$, we would obtain a similar partition of $R$. 
         That is, there exists $N'$ and a collection of ACVF' $k$ - definable subsets 
         $\{F'_1,\ldots,F'_{N'}\}$ which partition $R$ such that 
          if $M$ is a model of ACVF whose value group is $\mathbb{R}_{\infty}$ and 
     $\tau := (\mathbf{r},x) \in F'_j(M)$ then $\widehat{O}(M) \cap \widehat{V}(M)$ must have 
     $j$ connected components. 
         
          The two ACVF' partitions of $R$ can be used to give an ACVF' definable set $D' \subset R$ which is defined with parameters from $k$ such that 
          if  $M$ is a model of ACVF whose value group is $\mathbb{R}_{\infty}$ and 
     $\tau := (\mathbf{r},x) \in D'(M)$ then $N^M_{V'}(\mathbf{r},x) = dN^M_{V}(\mathbf{r},x)$.  
             By Beth's theorem, there exists an ACVF definable set $D \subset R$ defined with parameters in $k$ 
            such that if  $M$ is a model of ACVF whose value group is $\mathbb{R}_{\infty}$
        then $D(M) = D'(M)$ i.e. if $\tau \in D(M)$ then $N^M_{V'}(\mathbf{r},x) = dN^M_{V}(\mathbf{r},x)$.

   \end{proof}  

The following is a version of Theorem 1.7 for the spaces $\widehat{V}$. 

 \begin{thm}
    Let $\phi : V' \to V$ be a finite surjective morphism between irreducible, projective varieties with $V$ normal. 
 Let $g \in S$. There exists a pro-definable deformation retraction 
\begin{align*}
  \psi : I \times \widehat{V} \to \widehat{V}
\end{align*}
  which satisfies the following properties. 
\begin{enumerate}
  \item Let $I$ be a generalised interval of the form $[i,e]$. 
  The image $T_g := \psi(e,\widehat{V})$ of the deformation retraction $\psi$ is a $\Gamma$ - internal subset  of 
  $\widehat{V}$ \emph{[\cite{loes},Chapter 6]} and there exists a definable homeomorphism $j_g : T_g \to \Upsilon_g$ where
$\Upsilon_g \subset \Gamma_{\infty}^n$ is a $k$ - definable set. 
   \item There exists a well defined piecewise linear function $M_g : T_g \to \Gamma_{\infty}$ which satisfies the following conditions. 
The function $M_g$ takes values other than $\infty$. In fact there exists 
$x \in T_g(k)$ such that $M_g(x) \neq \infty$. 
 Let 
$\gamma \in T_g$ be a point for which $M_g(\gamma) \neq 0$ and 
$x \in \psi(e,\_)^{-1}(\gamma)$ such that 
there exists $L/k$ a complete non-Archimedean real valued algebraically closed field extension for which $\Gamma_{\infty}(L) = \mathbb{R}_{\infty}$ and
$x \in V(L)$.  There exists $W \in (g' \circ h)^{-1}(M_g(\gamma)) \cap \mathcal{O}_{x}^0$ such that the open set  
 $(\widehat{\phi})^{-1}(\widehat{W}(L) \cap \widehat{V}(L))  \subset \widehat{V}'(L)$ decomposes into the disjoint union of $\widehat{V}'(L)$ open sets, each
 homeomorphic to $\widehat{W}(L) \cap \widehat{V}(L)$ via $\widehat{\phi}$. Furthermore, let
 $O \in \mathcal{O}_{x}^0$ be such that
 $h(O) \in \mathbb{R}_\infty^{{n+1}^2}$ and  
  the preimage of $\widehat{O}(L) \cap \widehat{V}(L)$ under $\widehat{\phi}$
decomposes into 
the disjoint union of open sets in $\widehat{V}'(L)$, each 
 homeomorphic to $\widehat{O}(L)$ via the morphism $\widehat{\phi}$. Then $(g' \circ h)(O) \geq M_g(\gamma)$.
 \end{enumerate}
\end{thm}
 \begin{proof}
    By Lemmas 4.5 and 4.4, there exists a $k$ - definable subset $D$ of $R$ such that 
if $M$ is a model of ACVF with value group $\mathbb{R}_{\infty}$ then
$D(M)$ is the set of tuples $(\mathbf{r},x)$ defined over $M$
such that if $W \in \mathcal{O}^0_x(M)$ and $h(W) = \mathbf{r}$
 then 
 $\widehat{\phi}^{-1}(\widehat{W}(M) \cap \widehat{V}(M))$ is the disjoint union of $\widehat{V}(M)$ open sets 
  each of which are homeomorphic to $\widehat{W}(M) \cap \widehat{V}(M)$ via the morphism $\widehat{\phi}$. 
   The $k$ - definable set $D \subset [0,\infty]^{{n+1}^2} \times V$  
 comes equipped with a projection map $pr : D \to  V$.
  For $x \in V$, let $D_x := pr^{-1}(x)$.    
 By definition, 
$D_x$ is uniformly definable in $x$ with parameters in $k(x)$.  
Hence $g' \circ h(D_x) \subset \Gamma_{\infty}$ is a $k(x)$ - definable set. 
  For $x \in V$, let $g'_{\mathrm{inf}}(x)$ be the infimum of the 
definable set $g' \circ h(D_x)(\mathbb{U}) \subset \Gamma_{\infty}(\mathbb{U})$. We set 
$g'_{\mathrm{inf}}(x) = \infty$ when the set $D_x$ is empty.  
As $D_x$ is uniformly definable in $x$ with parameters in $k(x)$,
 $g'_{\mathrm{inf}}$ extends to a $k$ - definable function from $V \to \Gamma_{\infty}$ 
which can be extended to a pro - definable function  
 $\widehat{V}$ to $\Gamma_{\infty}$. 

 By Theorem 2.11, there exists a $k$ - definable $\Gamma$ - internal subset $T_g$ of 
 $\widehat{V}$, a pro - definable deformation retraction $H : I \times \widehat{V} \to \widehat{V}$ such that
 $H(e,\widehat{V}) = T_g$ where $e$ denotes the end point of the interval $I$ and
 the function $g'_{\mathrm{inf}}$ is constant along the fibres of the deformation retraction.
Let $M_g$ denote the restriction of $g'_{\mathrm{inf}}$ to $T_g$.    
 It remains to verify part (2) of the statement of the theorem. 

 Let $\gamma \in T_g$ and $x \in V$ such that $H(e,x) = \gamma$. Furthermore, we suppose that $L$ is a model of 
   ACVF over which $x$ is defined and $\Gamma_{\infty}(L) = \mathbb{R}_{\infty}$. Let $O \in \mathcal{O}^0_x(L)$ be such that 
   the preimage of $\widehat{O}(L) \cap \widehat{V}(L)$ under $\widehat{\phi}$
decomposes into 
the disjoint union of open sets in $\widehat{V}'(L)$, each 
 homeomorphic to $\widehat{O}(L)$ via the morphism $\widehat{\phi}$. It follows from Lemma 4.3 that 
$h(O) \in D_x(L)$. Hence from the definition of $g'_{\mathrm{inf}}(x)$ we get that 
 $(g' \circ h)(O) \geq g'_{\mathrm{inf}}(x) = M_g(\gamma)$. 
     Furthermore, by Lemma 4.7  
       there exists $W \in \mathcal{O}_x^0(L)$ such that $g' \circ h(W) = g'_{\mathrm{inf}}(x)$. This proves 
       part (2) of the theorem.  
     
       We now show that the function $M_g$ takes 
     values other than $\infty$. By Lemma 4.8,
     there exists affine open sets $U \subset V$, $U' \subset V'$ and an affine scheme $U''$ along 
  with morphisms $\phi_1 : U' \to U''$ purely inseparable and $\phi_2 : U'' \to U$ separable of degree $d$ such that
  the restriction of the morphism $\phi$ to $U'$ factors as $\phi_2 \circ \phi_1$.  
  As $\phi_1$ is purely inseparable, the induced map $\widehat{\phi_1}$ is a homeomorphism.    
   There 
  exists a smooth open sub scheme $U_0 \subset U$ over which the morphism $\phi_2$ is etale.
  Let $x$ be a $k$ - point in $U_0$. By Lemma 7.4.1 in \cite{loes}, there exists $O \in \mathcal{O}^0_x(\mathbb{R}_{\infty})$
 with $h(O) \in D_x$ with finite poly radius  
 which implies that $g'_{\mathrm{inf}}(x) < \infty$. 
      
        \end{proof} 

  \begin{lem} 
     Let $x \in V$ and $M$ be a model of ACVF which contains $k$ such that 
     $x \in V(M)$  and $\Gamma(M) \subseteq \mathbb{R}_{\infty}$. 
     Let $\mathbf{M}$ denote the structure defined by $(M,\mathbb{R}_{\infty})$. 
       There exists $O \in D_x(\mathbf{M})$ i.e $h(O) \in D_x(\mathbf{M})$ such that $g' \circ h(O) =  \mathrm{inf}\{g' \circ h(O) | O \in D_x(\mathbf{M})\}$ i.e. 
     the set  $\{g' \circ h(O) | O \in D_x(\mathbf{M})\}$ contains its infimum. 
  \end{lem} 
  \begin{proof} 
  Firstly observe that $D_x(\mathbf{M}) = D_x(\mathbb{R}_{\infty})$. 
     By definition, the function $g' : \mathbb{R}_{\infty}^{(n+1)^2} \to \mathbb{R}_{\infty}$
      is continuous with respect to the topology on $\mathbb{R}_{\infty}$ induced by its ordering. 
       Hence to prove the lemma it suffice to show that $\{h(O) | O \in D_x(\mathbf{M})\} \cup (\infty,\ldots,\infty)$ is compact. 
       As $\{h(O) | O \in D_x(\mathbf{M})\} \cup (\infty,\ldots,\infty) \subseteq [0,\infty]^{{n+1}^2}$, we need only show that it is closed.    
       Let $(O_n)_n$ be a sequence of elements in $D_x(\mathbf{M})$ such that $(h(O_n))_n$ converges to 
       $\mathbf{r} \in \mathbb{R}^{{n+1}^2}_{\infty}$. By definition of the family $\mathcal{O}_x^0(\mathbf{M})$, it can be verified that there exists an element 
       $W \in \mathcal{O}^0_x(\mathbf{M})$ such that $h(W) = \mathbf{r}$ and $W$ is uniquely determined by $\mathbf{r}$ and $x$.  
       We will show that $W \in D_x(\mathbf{M})$. There exists a subsequence $(O_{m_n})_n$ of 
       $(O_n)_n$ 
       such that the sequence of $\mathbb{R}_{\infty}$ - tuples $((r_{i,m_n})_i)_n := h(O_{m_n})_n$ is either increasing or decreasing at each component i.e. 
       the sequence $(r_{i,m_n})_n$ is either increasing or decreasing for every $i$.    
        Observe that if $O \in D_x(\mathbf{M})$ and $O' \in O_x(\mathbf{M})$ such that 
       $h(O') \geq h(O)$ then $O' \in D_x(\mathbf{M})$.        
       It follows that we can assume the sequence $(O_n)_n$ is increasing i.e. that the 
       sequence $(h(O)_n)_n$ is decreasing with respect to the point wise ordering. 
       This implies that $W = \bigcup_n O_n$ and it can be verified that the ball $W$ must also belong to $D_x(\mathbf{M})$.              
   \end{proof} 

   \begin{lem}
  Let $U$ and $U'$ be integral $k$-varieties and $\phi : U' \to U$ be a finite surjective morphism between them. 
There exists integral affine $k$-varieties 
$W \subset U$, $W' \subset U'$ and 
$W''$ along with morphisms  $\phi_1 : W' \to W''$ and 
$\phi_2 : W'' \to W$ such that 
\begin{enumerate}
\item  $W$ and $W'$ are Zariski open subsets 
of $U$ and $U'$ respectively.
\item  $\phi = \phi_2 \circ \phi_1$.  
\item  The extension of function fields $k(W'') \hookrightarrow k(W')$ induced by 
$\phi_1$ is purely inseparable. 
\item The extension of function fields $k(W) \hookrightarrow k(W'')$ induced by 
$\phi_2$ is separable. 
\end{enumerate}
 \end{lem}    
\begin{proof}
    To begin, observe that the morphism $\phi$ is flat over a Zariski open subset of $U$ and that
$U'$ is birational to its normalization. A flat morphism which is of finite type is open. 
It follows that there exists a Zariski open affine subset $W \subset U$ and a Zariski open set    
$W' \subset U'$ which is normal and in addition the restriction 
$\phi : W' \to W$ is flat and surjective. 

    The extension of function fields $k(W) \hookrightarrow k(W')$ can be realized as
    the composition of a purely inseparable extension and a separable extension. To be precise, there exists a 
    field $k(W'')$ such that $k(W) \hookrightarrow k(W')$ factorizes into 
      $k(W) \hookrightarrow k(W'')$ which is separable and 
      $k(W'') \hookrightarrow k(W')$ which is purely inseparable. 
 
 Let $A$ and $A'$ be $k$-algebras of finite type such that 
$W = \mathrm{Spec}(A)$ and $W' = \mathrm{Spec}(A')$.
Let $A''$ denote the normalisation of $A$ in $k(W'')$ and 
set $W'' := \mathrm{Spec}(A'')$. We hence have a separable morphism
$\phi_2 : W'' \to W$. 
 Since $A'$ was constructed to be integrally closed 
in $k(W')$ and to contain $A$, we have that the integral closure of $A$ in $k(W')$ must be contained in $A'$. This implies 
$A'$ contains $A''$ and hence we have a purely inseparable morphism 
$\phi_1 : W' \to W''$. 
    This proves the lemma. 
 \end{proof}

\section{Proof of the main theorem}

      In this section we prove Theorem $1.7$. 
Let $V$ and $V'$ be irreducible, projective $k$-varieties with $V$ normal and $\phi : V' \to V$ be a finite surjective morphism. The
morphism 
 $\phi$ induces a morphism between the respective analytifications. Hence we have
\begin{align*}
  \phi^{\mathrm{an}} : V'^{\mathrm{an}} \to V^{\mathrm{an}}.
\end{align*}

\begin{rem}  We introduced the collection of functions $S$ (Remark 1.5) for the following reason. 
Let $x \in V^{\mathrm{an}}(L)$ (Remark 1.1). 
Associated to $x$ is an $L$-point of $V_L^{\mathrm{an}}$ which we denoted $x_L$ (Section 1). 
In Section 2, we defined a collection of open neighbourhoods $\mathcal{O}_{x_L}$ of $x_L$ along with a 
function $h_L : \bigcup_{x \in V^{\mathrm{an}}(L)} \mathcal{O}_{x_L} \to \mathbb{R}_{\geq 0}^{(n+1)^2}$. An element $g \in S$ allows us to 
compare elements of the family $\mathcal{O}_{x_L}$. To be precise, the family 
$\bigcup_{x \in V^{\mathrm{an}}(L)} \mathcal{O}_{x_L}$ is partially ordered by the partial ordering defined by set theoretic inclusion. By Lemma 2.8, 
if $O_1, O_2 \in \bigcup_{x \in V^{\mathrm{an}}(L)} \mathcal{O}_{x_L}$ with $O_1 \subseteq O_2$ then 
$g \circ h_L(O_1) \leq g \circ h_L(O_2)$.

    We could also avoid using the function $g$ and instead do the following. Let $S'$ denote the collection of $0$ - definable total orderings of the set 
$(\mathbb{R}_{\infty})^{(n+1)^2}$ which satisfy the following property. If $\leq_p \in S'$ 
then given a pair of $(n+1)^2$-tuples $(x_i)_i$, $(y_j)_j$ such that 
$x_i \leq y_i$ for all $i$, we must have that $(x_i)_i \leq_p (y_i)_i$. We ask in addition that if $C$ is a non empty, compact 
subspace of $(\mathbb{R}_{\geq 0})^{(n+1)^2}$ then it contains a supremum with respect to the ordering $\leq_p$. 

                One can prove the following version of the main result: 
                
               \begin{thm}
    Let $\phi : V' \to V$ be a finite surjective morphism between irreducible, projective varieties with $V$ normal. 
 Let $\leq_g \in S'$. There exists a generalised real interval $I := [i,e]$
  and a deformation retraction 
\begin{align*}
  \psi : I \times V^{\mathrm{an}} \to V^{\mathrm{an}}
\end{align*}
  which satisfies the following properties. 
\begin{enumerate}
  \item The image $\psi(e,V^{\mathrm{an}})$ of the deformation retraction 
  $\psi$ is a finite simplicial complex. Let $\Upsilon_g$
 denote this finite simplicial complex.
   \item There exists a well defined function $M_g : \Upsilon_g \to \mathbb{R}_{\geq 0}$ 
   which satisfies the following conditions. 
The function $M_g$ is not identically zero and $\mathrm{log}(M_g)$ is piecewise linear. Let  
$\gamma \in \Upsilon_g$ be a point on the finite simplicial complex for which $M_g(\gamma) \neq 0$ and 
$x \in \psi(e,\_)^{-1}(\gamma)$. 
Let $L/k$ be any complete non-Archimedean real valued algebraically closed field extension such that 
$x \in V^{\mathrm{an}}(L)$. There exists $W_{x_L} \in (h_L)^{-1}(M_g(\gamma)) \cap \mathcal{O}_{x_L}$ 
such that the open set 
 $(\phi^{\mathrm{an}}_L)^{-1}(W_{x_L} \cap V_L^{\mathrm{an}})  \subset V_L'^{\mathrm{an}}$ decomposes into the disjoint union of open 
 sets, each
 homeomorphic to $W_{x_L} \cap V_L^{\mathrm{an}}$ via $\phi_L^{\mathrm{an}}$. Furthermore, let
 $O \in \mathcal{O}_{x_L}$ be such that the preimage
 of $O \cap V_L^{\mathrm{an}}$
  under $\phi_L^{\mathrm{an}}$
decomposes into 
the disjoint union of open sets in $V_L'^{\mathrm{an}}$, each 
 homeomorphic to $O \cap V_L^{\mathrm{an}}$ via the morphism $\phi_L^{\mathrm{an}}$. Then $(g \circ h_L)(O) \leq_g M_g(\gamma)$.
 \end{enumerate}
\end{thm}
\end{rem}
   
   The proof of the above result is similar to the proof of Theorem 1.7. We now 
   state and prove Theorem 1.7. We make use of Theorem 4.6 wherein we viewed the value group additively and 
   used the functions $h$ in place of $h_L$.   
  By definition, if $x \in V_L(L)$ where $L$ is a real valued model of ACVF
   and $O \in \mathcal{O}^0_x(L)$ then 
   $B_{\mathbf{L}}(O) \in \mathcal{O}_x$ where $B_{\mathbf{L}}(O)$ is 
   the Berkovich analytification of $O$ [Section 2.3].  
   We have that $h(O) = \mathrm{log}(h_L(B_{\mathbf{L}}(O)))$ where the function 
   $\mathrm{log}$ (Remark 1.3) is applied component wise. 
     \\

\noindent $\mathbf{Theorem}$ $\mathbf{1.7}$.\emph{ 
    Let $\phi : V' \to V$ be a finite surjective morphism between irreducible, projective varieties with 
    $V$ normal. 
 Let $g \in S$. There exists a generalised real interval $I := [i,e]$
  and a deformation retraction 
\begin{align*}
  \psi : I \times V^{\mathrm{an}} \to V^{\mathrm{an}}
\end{align*}
  which satisfies the following properties. 
\begin{enumerate}
  \item The image $\psi(e,V^{\mathrm{an}}) \subset V^{\mathrm{an}}$ of the deformation retraction 
  $\psi$ is homeomorphic to a finite simplicial complex. Let $\Upsilon_g$
 denote this finite simplicial complex.
   \item There exists a well defined 
    function $M_g : \Upsilon_g \to \mathbb{R}_{≥ 0}$ which satisfies the following conditions. 
The function $M_g$ takes values other $0$ and $\mathrm{log} \circ M_g$ is piecewise linear (Remark 1.3). Let 
$\gamma \in \Upsilon_g$ be a point on the finite simplicial complex for which $M_g(\gamma) \neq 0$ and 
$x \in \psi(e,\_)^{-1}(\gamma)$. 
Let $L/k$ be any complete non-Archimedean real valued algebraically closed field extension such that 
$x \in V^{\mathrm{an}}(L)$. There exists $W_{x_L} \in (g \circ h_L)^{-1}(M_g(\gamma)) \cap \mathcal{O}_{x_L}$ 
such that  
 $(\phi^{\mathrm{an}}_L)^{-1}(W_{x_L} \cap V_L^{\mathrm{an}})  \subset V_L'^{\mathrm{an}}$ decomposes into the disjoint union of open sets, each
 homeomorphic to $W_{x_L} \cap V_L^{\mathrm{an}}$ via $\phi_L^{\mathrm{an}}$. Furthermore, let
 $O \in \mathcal{O}_{x_L}$ be such that the preimage
 of $O \cap V_L^{\mathrm{an}}$ 
  under $\phi_L^{\mathrm{an}}$
decomposes into 
the disjoint union of open sets in $V_L'^{\mathrm{an}}$, each 
 homeomorphic to $O \cap V_L^{\mathrm{an}}$ via the morphism $\phi_L^{\mathrm{an}}$. Then $(g \circ h_L)(O) \leq M_g(\gamma)$.
 \end{enumerate}
}
\begin{proof} 
    We apply Theorem 4.6 to the given data. Hence there exists a pro $k$ - definable deformation retraction 
    $H : I \times \widehat{V} \to \widehat{V}$ where $I$ is a generalised interval defined over $k$,  
    a $k$ - definable $\Gamma$-internal set $Z \subset \widehat{V}$ which is $k$ - definably homeomorphic
     to a finite simplicial complex $\Upsilon'_g$ and is the image of the deformation retraction $H$ i.e. 
     $H(e,\widehat{V}) = Z$. Furthermore, there exists a $k$ - definable function $M'_g$ on $Z$ and hence 
     a piecewise linear function on $\Upsilon'_g$ which satisfies properties (1) and (2) stated in Theorem 4.6.  
     
     Let $\mathbf{k}$ denote the substructure of ACVF defined by the pair $(k,\mathbb{R}_{\infty})$. By section 2, the space 
     $B_{\mathbf{k}}(V)$ of weakly orthogonal $\mathbf{k}$ - types is canonically homeomorphic to the Berkovich space $V^{\mathrm{an}}$.
     We will for the remainder of this proof use the notation $B_{\mathbf{k}}(V)$ for the Berkovich space $V^{\mathrm{an}}$.  
     
     Given a valued field $M$ whose value group is contained in $\mathbb{R}_{\infty}$, there exists 
      a maximally complete valued field $K$ which contains $M $ and whose residue field is equal to the algebraic closure of 
     the residue field of $M$. By Kaplansky's theorem this field is unique up to isomorphism over $\mathbf{M} = (M, \mathbb{R}_{\infty})$ and we denote it $M_{max}$. 
     By Lemma 13.1.1 and Corollary 13.1.6 in \cite{loes}, 
    there exists a canonical continuous closed surjection $\widehat{V}(k_{max}) \to B_{\mathbf{k}}(V)$ which induces a 
    deformation retraction $H : I(\mathbf{k}) \times B_{\mathbf{k}}(V) \to B_{\mathbf{k}}(V)$ with image $Z(\mathbf{k})$.  
    As $Z(\mathbf{k})$ is homeomorphic to $\Upsilon_g(\mathbf{k})$ we identify it via this homeomorphism and set 
    $\Upsilon_g := Z(\mathbf{k})$. 
        Observe that $I(\mathbf{k})$ is a generalised real interval. By Theorem 4.6, the function $M'_g$ restricted to $\Upsilon_g$ takes 
    values different from $\infty$. It follows that $M_g := \mathrm{exp} \circ M'_g$ (Remark 4.1) takes values different from $0$. 
    
    We verify part (2) of the theorem. 
       Let $p \in B_{\mathbf{k}}(V)(L)$ 
    where $L$ is a real valued complete model of ACVF.  
    This is equivalent to saying that
     $\mathcal{H}(p) \subseteq L$. This implies that $p$ when viewed as a weakly orthogonal $\mathbf{k}$ - type on $V$ admits a realisation defined over $L$. 
    The point $x_L \in V(L)$ is such a realisation. 
     Let $\gamma = H(e,x_L)$. By 4.6, there exists $W \in \mathcal{O}^0_{x_L}(L_{max})$ such that 
    $g' \circ h(W) =  M'_g(\gamma)$ and $h(O) \in D_x$ where $D_x$ is as in the proof of Theorem 4.6. Since $\Gamma(L_{max}) = \Gamma(\mathbf{L})$,
    $W$ is in fact $\mathbf{L}$ - definable. Hence we must have that $B_{\mathbf{L}}(W) \in \mathcal{O}_{x_L}$ and by Remark 4.1 
    $g \circ h_L(B_{\mathbf{L}}(W)) = M_g(\gamma)$.  
    We have the following commutative diagram. 
    
     \setlength{\unitlength}{1cm}
\begin{picture}(10,5)
\put(4,1){$B_{\mathbf{L}}(V')$}
\put(7.7,1){$B_{\mathbf{L}}(V)$}
\put(3.8,3.5){$\widehat{V}'(L_{max})$}
\put(7.5,3.5){$\widehat{V}(L_{max})$}
\put(4.4,3.3){\vector(0,-1){1.75}}
\put(8,3.3){\vector(0,-1){1.75}}
\put(5.1,1.1){\vector(1,0){2.5}}
\put(5.3,3.6){\vector(1,0){2}}
\put(5.8,0.7){$\phi_L^{\mathrm{an}}$}
\put(5.8,3.2){$\widehat{\phi}$}
\put(4.5,2.3){$\pi_{V'_L}$}
\put(7.2,2.3){$\pi_{V_L}$} 
\end{picture}        
    
    As $h(W) \in D_x$, 
     $\widehat{\phi}^{-1}(\widehat{W}(L_{max}) \cap \widehat{V}(L_{max}))$ is the disjoint union of open sets each of which are homeomorphic to 
    $\widehat{W}(L_{max}) \cap \widehat{V}(L_{max})$ via the morphism $\widehat{\phi}$.
    Let $d$ denote the separable degree of the morphism $\phi$.
    By Lemma 4.3, there exists $W'_1,\ldots,W'_d \subset \widehat{V}'$ such that the
    $W'_i(L_{max})$ are open in $\widehat{V'}(L_{max})$,
    $\widehat{\phi}^{-1}(\widehat{W}(L_{max}) \cap \widehat{V}(L_{max})) = \bigcup_i W'_i(L_{max})$ and the morphism $\widehat{\phi}$ restricts 
    to a homeomorphism from each of the $W'_i(L_{max})$ onto $\widehat{W}(L_{max}) \cap \widehat{V}(L_{max})$. 
    Let ${W'_i}^0$ denote the subset of simple points of $W'_i$. The set ${W'_i}^0$ is 
    an ind-definable set [\cite{loes},Section 2.2] and as $W'^0_i(L_{max})$ is in bijection with the definable set $W(L_{max}) \cap V(L_{max})$ via the definable map $\phi$, 
    we deduce by compactness that $W'^0_i(L_{max})$ and hence $W'^0_i$
    must be definable as well. Since the morphism $\phi$ is defined over $k$ and
    $W$ is defined over $\mathbf{L} = (L,\mathbf{R}_{\infty})$, it follows that 
    $W'^0_i$ is defined over $\mathbf{L}$.
    As all models of ACVF which contain $L_{max}$ are equivalent, we deduce that 
    the $W'^0_i$ are disjoint and $\phi$ restricts to a  bijection from $W'^0_i$ onto $W \cap V$. 
    Also $\widehat{W'^0_i} = W'_i$.
    It can be deduced from the definition of the morphism 
    $\pi_{V'_L}$ that it restricts to a morphism from $W'_i(L_{max})$ onto $B_{\mathbf{L}}({W'_i}^0)$ i.e. 
    $\pi_{V'_L}(W'_i(L_{max})) = B_{\mathbf{L}}(W'^0_i)$. We claim that the $B_{\mathbf{L}}(W'^0_i)$ are 
    disjoint open subspaces of $B_{\mathbf{L}}(V')$. That they are disjoint follows from the 
    fact that the $W'_i(L_{max})$ are disjoint. Indeed, let $p$ be an $\mathbf{L}$ - type
   lying in the intersection of $B_{\mathbf{L}}(W'^0_1)$ and $B_{\mathbf{L}}(W'^0_2)$ which are distinct. Let 
   $c$ be a realisation of $p$. The type $\mathrm{tp}(c|L_{max})$ is an $L_{max}$ - stably dominated type that belongs to 
   both $W'_1$ and $W'_2$ which is not possible. 
   We now show that for every $i$, $B_{\mathbf{L}}(W'^0_i)$ is an open subspace of 
   $B_{\mathbf{L}}(V')$. The morphism $\pi_{V'_L}$ is closed and hence it restricts to a closed surjection from 
   $\widehat{\phi}^{-1}(\widehat{W}(L_{max}) \cap \widehat{V}(L_{max}))$ onto 
   $\bigcup_i B_{\mathbf{L}}(W'^0_i)$. For a fixed $j$, the set $\bigcup_{i \neq j} W'_i(L_{max})$ is a closed 
   subspace of  $\widehat{\phi}^{-1}(\widehat{W}(L_{max}) \cap \widehat{V}(L_{max}))$ whose image via the morphism 
   $\pi_{V'_L}$ is the set $\bigcup_{i \neq j} B_{\mathbf{L}}(W'^0_i)$. Since the $B_{\mathbf{L}}(W'^0_i)$ are disjoint, the 
   set $B_{\mathbf{L}}(W'^0_j)$ is open in $\bigcup_i B_{\mathbf{L}}(W'^0_i)$. The commutative diagram implies that
   \begin{align*}
        (\phi_L^{\mathrm{an}})^{-1}(B_\mathbf{L}(W) \cap B_{\mathbf{L}}(V)) = \bigcup_i B_{\mathbf{L}}(W'^0_i). 
     \end{align*}
       Hence the $B_{\mathbf{L}}(W'^0_i)$ are open in $B_{\mathbf{L}}(V')$.

      We claim that $\phi^{\mathrm{an}}$ restricts to a 
    homeomorphism from each of the $B_{\mathbf{L}}(W'^0_i)$ onto $B_{\mathbf{L}}(W) \cap B_{\mathbf{L}}(V)$.
    We fix an index $j$. 
    Since the vertical arrows of the commutative diagram above are closed and the 
    restriction of $\widehat{\phi}$ to $W'_i(L_{max})$ is a homeomorphism onto 
    $\widehat{W}(L_{max}) \cap \widehat{V}(L_{max})$
    the morphism $\phi_L^{\mathrm{an}}$ is a closed surjection from
    $B_{\mathbf{L}}(W'^0_j)$ onto $B_{\mathbf{L}}(W) \cap B_{\mathbf{L}}(V)$. 
    We now show that it is also a bijection. Let $p \in B_{\mathbf{L}}(W) \cap B_{\mathbf{L}}(V)$ and $c$ be a realisation of $p$. 
    The point $c$ is simple in $\widehat{W} \cap \widehat{V}$. By Lemma 4.2, there exists exactly $d$ preimages of $c$ in $V'$ each contained in 
    exactly one $W'^0_i$. Let $\{c'_1,\ldots,c'_d\}$ denote this set of preimages where $c'_i \in W'^0_i$. The type 
    $\mathrm{tp}(c'_i|L_{max})$ is an $L_{max}$ - stably dominated type contained in $W'_i$ and its image in $\widehat{W} \cap \widehat{V}$ for the morphism 
    $\widehat{\phi}$ is the stably dominated type 
    $\mathrm{tp}(c|L_{max})$. As the $B_{\mathbf{L}}(W'^0_i)$ are mutually disjoint and 
    $\pi_{V'_L}(W'_i(L_{max})) = B_{\mathbf{L}}(W'^0_i)$
    it follows that there must be at least $d$ weakly orthogonal types in $B_{\mathbf{L}}(V')$ which map to $p$. However, 
    the cardinality of the fibre over $p$ for the morphism $\phi_L^{\mathrm{an}}$ is bounded above by $d$. It follows that 
     there exists one unique element in $B_{\mathbf{L}}(W'_i)$ which maps to $p$ via $\phi^{\mathrm{an}}$. This implies that the morphism 
     $\phi_L^{\mathrm{an}}$ restricts to a closed bijection from $B_{\mathbf{L}}(W'_i)$ onto $B_{\mathbf{L}}(W) \cap B_{\mathbf{L}}(V)$. It is hence a homeomorphism.   
     
      We now verify the remainder of the theorem. Let
      $O \in \mathcal{O}^0_{x_L}$ be such that $(\phi_L^{\mathrm{an}})^{-1}(B_{\mathbf{L}}(O) \cap B_{\mathbf{L}}(V))$
      is the disjoint union of open sets in $B_{\mathbf{L}}(V')$ each of which are homeomorphic 
      to $B_{\mathbf{L}}(O) \cap B_{\mathbf{L}}(V)$ via the morphism $\phi_L^{\mathrm{an}}$.  
      From the definition of the functions $h$ and $h_L$ and Remark 4.1,  
      $g' \circ h(O) = \mathrm{log}((g \circ h_L)(O))$. 
      It can be deduced from the definition of the functions 
      $g$, $g'$ and $M_g$ that 
      to complete the proof we must show that $g' \circ h(O) \geq M'_g(\gamma)$.  
      The field $L$ is algebraically closed and non - trivially valued. Hence 
      its value group $\Gamma(L)$ is dense in $\mathbb{R}_{\infty}$. This implies that 
      $\{h(O) | O \in  \mathcal{O}^0_{x_L}(L)\}$ which is the set of elements definable over $L$ is dense in 
      $\{h(O) | O \in  \mathcal{O}^0_{x_L}(\mathbf{L})\}$. As $g'(\mathbb{R}_{\infty})$ is a continuous function, we
       can reduce to when $O \in \mathcal{O}^0_{x_L}(L)$.   
       To show $g' \circ h(O) \geq M'_g(\gamma)$ it suffices to prove  
       that $h(O) \in D_{x_L}$. 
       By definition, $(\phi_L^{\mathrm{an}})^{-1}(B_{\mathbf{L}}(O) \cap B_{\mathbf{L}}(V))$ must be the disjoint union of open sets  
       in $B_{\mathbf{L}}(V')$. By Proposition 3.3, there exists $d$, $L$ - definable semi-algebraic sets $O'_i \subset V'$ such that 
       $(\phi_L^{\mathrm{an}})^{-1}(B_{\mathbf{L}}(O) \cap B_{\mathbf{L}}(V)) = \bigcup_i B_{\mathbf{L}}(O'_i)$ and the morphism 
       $\phi_L^{\mathrm{an}}$ restricts to a homeomorphism from $B_{\mathbf{L}}(O'_i)$ onto $B_{\mathbf{L}}(O) \cap B_{\mathbf{L}}(V)$ for every 
       $i$. 
         This implies in particular that $\phi$ restricts to a bijection from $O'_i(L)$ onto $O(L)$ for every $i$. As all 
          models of ACVF which contain $L$ are equivalent to $L$, we must have that the morphism $\phi$ restricts to a bijection between 
       $O'_i(\mathbb{U})$ and $O(\mathbb{U}) \cap V(\mathbb{U})$ for every $i$ and in addition 
       $O'_i(\mathbb{U}) \cap O'_j(\mathbb{U})$ is empty.  
       Furthermore, for any $z \in O \cap V$, there exists exactly $d$ preimages of $z$ in $V'$, exactly one in 
       each of the $O'_i$.   
       We now show that the morphism 
       $\widehat{\phi}$ induces a homeomorphism between $\widehat{O}_i$ and $\widehat{O} \cap \widehat{V}$.  
       Firstly, our description of the sets $O'_i$ from Proposition 3.3 was explicit, and it follows from this description that 
       $\widehat{O'_i}$ is an open subset of $\widehat{V'}$. 
       Since $\widehat{V}$ is normal, the morphism $\widehat{\phi}$ is open and the restriction of $\widehat{\phi}$ to the open set 
       $\widehat{O}'_i$ is also an open map. This restriction is in fact bijective. Indeed, let $p \in (\widehat{O} \cap \widehat{V})(L_{max})$. By definition $p$ is a stably dominated 
       type. Let $a$ be a realisation of the type $p_{|L}$. The arguments above imply that there exists exactly $d$ preimages of $a$, one in each of the 
       $O'_i$. But as $O'_i$ is defined over $L$ there exists at least $d$ preimages of $p$. However the cardinality of the set
       $\widehat{\phi}^{-1}(p)$ is bounded above by $d$. Hence there exists exactly one preimage of $p$ in each of the $\widehat{O'_i}$. It follows that the restriction 
       of $\widehat{\phi}$ to each of the $\widehat{O'_i}$ is a bijective open morphism which in turn implies that 
       the restriction of $\widehat{\phi}$ to $\widehat{O'_i}(L_{max})$ is a bijective open morphism onto $\widehat{O}(L_{max}) \cap \widehat{V}(L_{max})$.
       Hence
        $h(O) \in D_{x_L}$.   
\end{proof}

\section{The tying up of loose ends}
 
    In the introduction we announced that the goal of this article was to prove a generalization of Theorem 1.2.
     In the previous section we proved Theorem 1.7.
     We now show that the main theorem implies Theorem 1.2. 
    We begin by showing that 
   Theorem 1.2 is equivalent to Theorem 1.6. We write the value group multiplicatively in this section.  
    
\begin{prop}
Let $\phi : \mathbb{P}^1_k \to \mathbb{P}^1_k$ be a finite morphism. Given such a morphism, 
theorems $1.2$ and $1.6$ are equivalent.
\end{prop}
\begin{proof}
    Let us assume that Theorem $1.2$ is true. 
Let $x \in \mathbb{P}^{1,\mathrm{an}}_k$ and $L/k$ be a
complete, algebraically closed, non-Archimedean real valued field extension of $k$ such that 
$x \in \mathbb{P}^{1,\mathrm{an}}_k(L)$. By definition, $f(x)$ is
the minimum of the radius of the largest Berkovich open ball around $x_L$ whose preimage is the disjoint union of 
homeomorphic copies of the ball via the morphism $\phi_L^{\mathrm{an}}$
and $1$. The assumption that Theorem $1.2$ is true 
implies that there exists a finite simplicial complex
 $\Upsilon \subset \mathbb{P}^{1,\mathrm{an}}_k$ and a deformation retraction 
\begin{align*}
  \psi : I \times \mathbb{P}_k^{1,\mathrm{an}} \to \mathbb{P}_k^{1,\mathrm{an}}
\end{align*}
with image $\Upsilon$ such that the function $f$ is constant on the fibres of this retraction. 
We define $M : \Upsilon \to [0,1]$
 as follows. Let $\gamma \in \Upsilon$. Pick any $x \in \mathbb{P}^{1,\mathrm{an}}_k$ which retracts to $\gamma$ and set 
$M(\gamma) := f(x)$. Since the function $f$ is constant along the fibres of the retraction, $M$ is well defined. It is also not identically zero
and $\mathrm{log}(M)$ is piecewise linear.   
It can be checked that the existence of the simplicial complex $\Upsilon$, the deformation retraction $\psi$ and the 
function $M : \Upsilon \to [0,1]$ imply that Theorem $1.6$ is true.  
  
  We now assume Theorem $1.6$ and show $1.2$ is true. By assumption, there exists
a finite, simplicial complex
$\Upsilon \subset \mathbb{P}^{1,\mathrm{an}}_k$, a retraction 
\begin{align*}
  \psi : I \times \mathbb{P}_k^{1,\mathrm{an}} \to \mathbb{P}_k^{1,\mathrm{an}}
\end{align*}
with image $\Upsilon$ and a function $M : \Upsilon \to [0,1]$ such that if 
$x \in \psi(e,\_)^{-1}(\gamma)$ where $M(\gamma) > 0$ and 
$L/k$ is a
complete, algebraically closed, non-Archimedean real valued field extension of $k$ such that 
$x \in \mathbb{P}^{1,\mathrm{an}}_k(L)$ then
the Berkovich open ball around $x_L$ of radius $M(\gamma)$ 
 decomposes into the disjoint union of Berkovich open balls each
 homeomorphic to it. Furthermore, if $O$ is any other Berkovich open ball around $x_L$ whose radius is 
 less than or equal to $1$ such that its 
preimage for the morphism $\phi_L^{\mathrm{an}}$ decomposes into the disjoint union of homeomorphic copies of the ball
 then its radius is less than or equal to $M(\gamma)$. If $\psi(e,x) = \gamma$ then it is clear that $f(x) = M(\gamma)$. 
Hence the function $f$ is constant along the fibres of the retraction morphism $\psi(e,\_)$. 
Furthermore, $f$ is not identically zero and $\mathrm{log}(f)$ is piecewise linear on $\Upsilon$.
This proves Theorem $1.2$.  
\end{proof}

\begin{prop}
 Theorem 1.7 implies Theorem 1.6. 
 \end{prop}
 \begin{proof}
 To apply Theorem 1.7, we need to choose a suitable
     definable
 function $g : \mathbb{R}^{4}_{\geq 0} \to \mathbb{R}_{\geq 0}$. 
 Let $(r_1,..,r_4) \in \mathbb{R}^{4}_{\geq 0}$. We set $g(r_1,..,r_4) := \Pi_i r_i$. By Theorem $1.7$, there 
 exists a finite simplicial complex $\Upsilon'$, a deformation retraction 
 $\psi' : [i,e] \times \mathbb{P}_k^{1,\mathrm{an}} \to \mathbb{P}_k^{1,\mathrm{an}}$ with image $\Upsilon'$
  and a function 
 $M' : \Upsilon' \to \mathbb{R}_{\geq 0}$ which satisfies the following property.
 Let $L/k$ be a non-Archimedean real valued field extension of $k$ and $x \in \mathbb{P}_k^{1,\mathrm{an}}(L)$. 
 Let $\gamma = \psi'(e,x_L)$. There exists $O \in \mathcal{O}_{x_L}$ such that 
 $(g \circ h_L)(O) = M'(\gamma)$ and the open set $(\phi_L^{\mathrm{an}})^{-1}(O)$ decomposes into the disjoint 
 union of homeomorphic copies of $O$ via the morphism $\phi_L^{\mathrm{an}}$.
 
 Let $x_L$ have homogenous coordinates $[a:1]$. By 2.2.1, if $|a| \leq 1$ then 
 the family $\mathcal{O}_{x_L}$ is the set of Berkovich open balls around $x_L$ whose radius is bounded by $1$. 
 If $|a| > 1$ then $\mathcal{O}_{x_L}$ contains the set of Berkovich open balls around $x_L$. 
 As sketched in 2.2.1, the radius of these Berkovich open balls can be expressed in terms of the 
 $4$-tuple $h_L(O)$.  
    
     Using $2.2.1$ we see that if $O \in \mathcal{O}_{x_L}$ is a Berkovich open ball $B(x_L,r)$ then 
the formula which relates the radius $r$ to the tuple $h_L(O)$ varies according to the value $|a|$.       
 For this reason we modify the simplicial complex suitably.  
Let $Z_0$ be the smallest path-connected closed subspace that contains the set $\{0, \infty\}$ and 
$\Upsilon$ be a finite simplicial complex that contains  $\Upsilon' \cup Z_0$. The space $\mathbb{P}^{1,\mathrm{an}}_k$  
 admits a deformation retraction onto any finite sub graph. In particular there exists a deformation retraction 
 $\psi : [i,e] \times \mathbb{P}_k^{1,\mathrm{an}} \to \mathbb{P}_k^{1,\mathrm{an}}$ with image $\Upsilon$. The function $M'$ extends to a 
 function on $\Upsilon$ as follows. Let $p \in \Upsilon$. We set 
 $M''(p) := M'(\psi'(e,p))$. The function $M'' : \Upsilon \to \mathbb{R}_{\geq 0}$ is well   
 defined. 
 
    We now define $M : \Upsilon \to \mathbb{R}_{\geq 0}$ which will imply Theorem 1.6. 
 \begin{align*}
                             &  M''(p)  &|T_1(p)| < |T_2(p)| \\
  M(p)   := \Bigg\{   &  M''(p)^{1/2}  &|T_1(p)| = |T_2(p)| \\
                             &   \mathrm{Min}\{1, M''(p)(|T_1(p)|/|T_2(p)|)^2 \} & |T_1(p)| > |T_2(p)| 
 \end{align*}     
      
     Using $2.2.1$ it can be verified that the function $M$ is bounded above by $1$. 
     It remains to check that the function $M$ defined above satisfies the properties required by Theorem 1.6. 
Let $x_L \in \mathbb{P}_L^{1,\mathrm{an}}(L)$ have homogenous coordinates $[a:1]$ and 
let $x_L$ retract to the point $p \in \Upsilon$ via the retraction $\psi$. 

     Let $|a| > 1$. 
     By Theorem 1.7, there exists $O \in \mathcal{O}_{x_L}$ such that the preimage of $O$ 
     is the disjoint union of copies of $O$ for the morphism $\phi_L^{\mathrm{an}}$ and also that 
     $g \circ h_L$ achieves its maximal value at $O$
      amongst all elements of $\mathcal{O}_{x_L}$ which satisfy this property.
      Let $h_L(O) = ((1,r),(1,1))$. 
      It follows that that $M''(p) = r$. 
     Observe that since $x_L$ retracts to the point $p$ via the retraction $\psi$, $|(T_1/T_2)(p)| = |a|$.

         If $M(p) = 1$ we must show that
      the preimage of the Berkovich open ball
     $B(x_L,1)$ decomposes into the disjoint union of copies of itself.
       It follows from the definition of the function $M(p)$ that 
      $r \geq  1/|a|^2$.  Any $O' \in \mathcal{O}_{x_L}$
       such that $h_L(O') = ((1,s),(1,1))$ with $s \leq r$ must be such that  
      its preimage for the morphism $\phi_L^{\mathrm{an}}$ is the disjoint union of homeomorphic copies of itself. 
      In particular we may choose $O'$ for which $h_L(O') = ((1,1/|a|^2),(1,1))$.  
      By $2.2.1$ the open neighbourhood 
      $O'$ is a Berkovich open ball 
      around $x_L$ of radius $1$. 
       
             Let $M(p) < 1$.
       By definition of the function $M$ we have that 
      $r < 1/|a|^2$. Using 2.2.1, we see that the open 
      set $O$ corresponds to the Berkovich open ball around $x_L$ of radius 
      $r|a|^2$. Let $B(x_L,s)$ be a Berkovich open ball around $x_L$ such that its preimage decomposes into 
      the disjoint union of homeomorphic copies of itself via the morphism $\phi_L^{\mathrm{an}}$. By $2.2.1$, we have that   
      $h_L(B(x_L,s)) = ((1,s/|a|^2),(1,1))$. Theorem 1.7 then implies that $s \leq r|a|^2$. 
      
        If $|a| \leq 1$ then
from our construction of $\Upsilon$ the point
 $x_L$ must retract to $p \in \Upsilon$ such that $|T_1(p)| \leq |T_2(p)|$. 
As done above, by our choice of $g$, the calculations in Section 2.2.1 and Theorem 1.7, it can be shown that 
the preimage of the Berkovich open ball $B(x_L, M(p))$ for the morphism 
$\phi_L^{\mathrm{an}}$ decomposes into the disjoint union of homeomorphic copies of $B(x_L,M(p))$
via the morphism $\phi_L^{\mathrm{an}}$. Furthermore, if $B(x_L,s)$ is a Berkovich open ball such that its preimage 
splits into the disjoint union of homeomorphic copies of $B(x_L,s)$ via $\phi_L^{\mathrm{an}}$ then by 2.2.1,
$s \leq M(p)$.   
 
       That the function $\mathrm{log}(M)$ is piecewise linear on $\Upsilon$ follows from the fact that the function $\mathrm{log}(M')$ is piecewise linear on 
       the finite graph $\Upsilon'$.  
              \end{proof}

\textbf{Acknowledgments:} Firstly, I would like to thank my advisor 
 Professor François Loeser for his support and guidance during this period of work. I would also like to thank
Ariyan Javanpeykar, Giovanni Rosso and Yimu Yin for their suggestions and comments. I am grateful to 
J\'er\^ome Poineau and the referee whose suggestions 
extended the scope of the principal results and improved the quality of the paper.

\end{document}